\documentclass[10pt]{amsart}
\numberwithin{equation}{section}
\usepackage{amsmath}
\usepackage{amscd,amsthm,amssymb,amsfonts}
\usepackage{mathrsfs}
\usepackage{dsfont}
\usepackage{stmaryrd}
\usepackage{euscript}
\usepackage{expdlist}
\usepackage{enumerate}

\input xy
\newtheorem{thm}{Theorem}[section]

\newtheorem*{thm*}{Theorem}

\newtheorem{lm}[thm]{Lemma}
\newtheorem{cor}[thm]{Corollary}
\newtheorem*{cor*}{Corollary}
\newtheorem{prop}[thm]{Proposition}
\newtheorem*{conj*}{Conjecture}

\theoremstyle{remark}

\theoremstyle{definition}
\newtheorem*{defn*}{Definition}
\newtheorem{Remark}[thm]{Remark}
\newtheorem{I_Remark*}{Remark}

\newcommand{\nc}{\newcommand}

\newcommand{\beq}{\begin{equation}}
\newcommand{\eeq}{\end{equation}}
\newcommand{\bpmx}{\begin{pmatrix}}
\newcommand{\epmx}{\end{pmatrix}}
\newcommand{\bbmx}{\begin{bmatrix}}
\newcommand{\ebmx}{\end{bmatrix}}

\def\parref#1{\ref{#1}}
\def\thmref#1{Theorem~\parref{#1}}

\def\propref#1{Proposition~\parref{#1}}
\def\corref#1{Corollary~\parref{#1}}

\def\lmref#1{Lemma~\parref{#1}}

\def\makeop#1{\expandafter\def\csname#1\endcsname
  {\mathop{\rm #1}\nolimits}\ignorespaces}
\makeop{Hom}   \makeop{End}   \makeop{Aut}   %\makeop{Isom}
\makeop{Pic} \makeop{Gal}       \makeop{Div} \makeop{Lie}
\makeop{PGL}   \makeop{Corr} \makeop{PSL} \makeop{sgn} \makeop{Spf}
 \makeop{Tr} \makeop{Nr} \makeop{Fr} \makeop{disc}
\makeop{Proj} \makeop{supp} \makeop{ker}   \makeop{Im} \makeop{dom}
\makeop{coker} \makeop{Stab} \makeop{SO} \makeop{SL} \makeop{SL}
\makeop{Cl}    \makeop{cond} \makeop{Br} \makeop{inv} \makeop{rank}
\makeop{id}    \makeop{Fil} \makeop{Frac}  \makeop{GL} \makeop{SU}
\makeop{Trd}   \makeop{Sp} \makeop{Tr}    \makeop{Trd} \makeop{Res}
\makeop{ind} \makeop{depth} \makeop{Tr} \makeop{st} \makeop{Ad}
\makeop{Int} \makeop{tr}    \makeop{Sym} \makeop{can} \makeop{SO}
\makeop{torsion} \makeop{GSp} \makeop{Tor}\makeop{Ker} \makeop{rec}
\makeop{Ind} \makeop{Coker}
 \makeop{vol} \makeop{Ext} \makeop{gr} \makeop{ad}
 \makeop{Gr}\makeop{corank} \makeop{Ann}
\makeop{Hol} %Holomorphic
\makeop{Fitt} \makeop{Mp} \makeop{CAP}

\def\makebb#1{\expandafter\def
  \csname bb#1\endcsname{{\mathbb{#1}}}\ignorespaces}
\def\makebf#1{\expandafter\def\csname bf#1\endcsname{{\bf
      #1}}\ignorespaces}
\def\makegr#1{\expandafter\def
  \csname gr#1\endcsname{{\mathfrak{#1}}}\ignorespaces}
\def\makescr#1{\expandafter\def
  \csname scr#1\endcsname{{\EuScript{#1}}}\ignorespaces}
\def\makecal#1{\expandafter\def\csname cal#1\endcsname{{\mathcal
      #1}}\ignorespaces}
\def\doLetters#1{#1A #1B #1C #1D #1E #1F #1G #1H #1I #1J #1K #1L #1M
                 #1N #1O #1P #1Q #1R #1S #1T #1U #1V #1W #1X #1Y #1Z}
\def\doletters#1{#1a #1b #1c #1d #1e #1f #1g #1h #1i #1j #1k #1l #1m
                 #1n #1o #1p #1q #1r #1s #1t #1u #1v #1w #1x #1y #1z}
\doLetters\makebb   \doLetters\makecal  \doLetters\makebf
\doLetters\makescr
\doletters\makebf   \doLetters\makegr   \doletters\makegr

    \def\setminus{\smallsetminus}

\normalsize

\makeop{Ram} \makeop{Rep} \makeop{mass}

\makeop{Bl}

\def\cB{\EuScript B}

\def\cL{{\mathcal L}}

\def\cS{{\mathcal S}}

\def\cV{{\mathcal V}}

\def\ot{\otimes}

\def\hookto{\hookrightarrow}
\def\longto{\longrightarrow}
\def\ol{\overline}  \nc{\opp}{\mathrm{opp}} \nc{\ul}{\underline}

\def\XYmatrix{\xymatrix@M=8pt} % make \xymatrix not too cluttered
\def\ncmd{\newcommand}
\ncmd{\xysubset}[1][r]{\ar@<-2.5pt>@{^(-}[#1]\ar@<2.5pt>@{_(-}[#1]}
\ncmd{\XYmatrixc}[1]{\vcenter{\XYmatrix{#1}}}
\ncmd{\xyto}[1][r]{\ar@{->}[#1]}
\ncmd{\xyinj}[1][r]{\ar@{^(->}[#1]}
\ncmd{\xysurj}[1][r]{\ar@{->>}[#1]}
\ncmd{\xyline}[1][r]{\ar@{-}[#1]}
\ncmd{\xydotsto}[1][r]{\ar@{.>}[#1]}
\ncmd{\xydots}[1][r]{\ar@{.}[#1]}
\ncmd{\xyleadsto}[1][r]{\ar@{~>}[#1]}
\ncmd{\xyeq}[1][r]{\ar@{=}[#1]} \ncmd{\xyequal}[1][r]{\ar@{=}[#1]}
\ncmd{\xyequals}[1][r]{\ar@{=}[#1]}
\ncmd{\xymapsto}[1][r]{l\ar@{|->}[#1]}\ncmd{\xyimplies}[1][r]{\ar@{=>}[#1]}
\ncmd{\xyiso}{\ar[r]_-{\sim}}
\def\injxy{\ar@{^(->}}

\newcommand{\pMX}[4]{\begin{pmatrix}
{#1}& {#2}\\
{#3}&{#4}\end{pmatrix} }

\newcommand{\seesaw}[4]{{#1}\ar@{-}[rd]\ar@{-}[d]&{#2}\ar@{-}[d]\\
{#3}\ar@{-}[ru]&{#4}}

\def\x{{\times}}

\newcommand\stt[1]{\left\{#1\right\}}

\renewcommand\Re{\text{Re}\,}
\usepackage{color}
\usepackage{hyperref}
\usepackage{MnSymbol}
\usepackage{mathdots}
\usepackage{tikz-cd}
\usepackage{adjustbox}
\newcommand{\BigWedge}{\mathord{\adjustbox{valign=B,totalheight=.6\baselineskip}{$\bigwedge$}}}

\def\SO{{\rm SO}}
\def\GL{{\rm GL}}

\def\Mat{{\rm Mat}}
\def\b{\bar}
\def\t{\tilde}
\def\h{\hat}
\def\a{\alpha}
\def\bt{\boxtimes}
\def\la{\lambda}
\def\Wedge{\BigWedge}

\makeatletter
\@namedef{subjclassname@2020}{\textup{2020} Mathematics Subject Classification}
\makeatother

\title{On the Rankin-Selberg $L$-factors for $\SO_{5}\x\GL_2$}
\author{Yao Cheng}
\address{No. 151, Yingzhuan Road, Tamsui District, New Taipei City 251, Taiwan (R.O.C),  Lui-Hsien Memorial 
Science Hall.}
\email{briancheng@mail.tku.edu.tw}
\begin{document}
\maketitle
\begin{abstract}
Let $\pi$ and $\tau$ be a smooth generic representation of $\SO_5$ and $\GL_2$ respectively over 
a non-archimedean local field. Assume that $\pi$ is irreducible and $\tau$ is irreducible or induced of Langlands' type. 
We show that the $L$- and $\epsilon$-factors attached to $\pi\x\tau$ defined by the Rankin-Selberg integrals 
and the associated Weil-Deligne representation coincide. Similar compatibility results are also obtained for the local 
factors defined through the Novodvorsky's local zeta integrals that attached to generic representations of 
${\rm GSp}_4\x\GL_2$.
\end{abstract}

\section{Introduction}
Let $F$ be a finite extension over $\bbQ_p$ and $\psi$ be a non-trivial additive character of $F$. 
Let $\pi$ and $\tau$ be an irreducible smooth generic representation of split $\SO_{2n+1}(F)$ and $\GL_r(F)$ respectively. 
We denote by $\phi_\pi$ the $L$-parameter attached to $\pi$ defined by Jiang-Soudry 
(\cite{JiangSoudry2003}, \cite{JiangSoudry2004}) and $\phi_\tau$ the $L$-parameter attached 
to $\tau$ under the local Langlands correspondence for $\GL_r$ 
(\cite{HarrisTaylor2001}, \cite{Henniart2000}, \cite{Scholze2013}). Then to $\pi\x\tau$ and $\psi$, one can define the (local) 
$L$-, $\epsilon$- and $\gamma$-factors through the following approaches: 
\begin{itemize}
\item (WD) the associated Weil-Delinge representation (\cite{Tate1979}); 
\item (LS) the Langlands-Shahidi method (\cite{Shahidi1990});
\item (RS) the Rankin-Selberg integrals (\cite{Soudry1993}). \quad\\
\end{itemize}

While it is natural to ask whether these approaches define the same local factors, the equalities among these local 
factors are actually essential for studying analytic properties of Langlands' automorphic $L$-functions and have
applications to the theory of automorphic representations. For instance, in a different setting, Yamana (\cite{Yamana2014})
proved that the $L$-factors defined in \cite{GPSR1987}, \cite{PSR1986} and \cite{LapidRallis2005} agree.
Consequently, he was able to show that non-vanishing of global theta liftings is characterized in terms of the analytic 
properties of the $complete$ $L$-functions and the occurrence in the local theta correspondence. As another example, 
Ikeda (\cite{Ikeda1992}) proved that the $L$-factors defined in \cite{PSR1987} agree with the one defined by the 
associated Weil-Deligne representations, provided that the representations are local components of global 
representations. As a result, he was able to determine the poles of the $triple$ $product$ $L$-functions and 
supply some evidents that consistent with "the Langlands' philosophy".\\

In the literature, Jiang-Soudry proved that the local factors attached to $\pi\x\tau$ and $\psi$ defined by 
(WD) and (LS) coincide in \cite{JiangSoudry2004}. On the other hand, the $\gamma$-factors attached to 
$\pi\x\tau$ and $\psi$ defined by (LS) and (RS) also essentially coincide according to \cite{Soudry2000} and 
\cite{Kaplan2015}. Therefore, to settle the question arising in the first paragraph, it remains to show that the $L$-factors 
attached to $\pi\x\tau$ given by (RS) agree with the one given by (WD) or (LS). However, this question is usually more 
involved, for the $L$-factors defined by (LS) are designed to have right properties with respect to the local Langlands 
correspondence, whereas the $L$-factors defined by (RS) are designed to control poles the local Rankin-Selberg 
integrals. While both approaches give essentially the same $\gamma$-factors, it is not at all clear that they give the 
same $L$-factors. In this note, we resolve this question when $n=r=2$. In fact, our results allow 
$\tau$ to be reducible, and in addition, we also obtain similar results for the closely related local integrals attached to 
generic representations of ${\rm GSp}_4\x\GL_2$ introduced by Novodvorsky in \cite{Novodvorsky1979}.

\subsection{Main Results}\label{SS:main result}
Let us describe our main results in this subsection. 

\subsubsection{Local factors via (WD) and (RS)}
Using $\phi_\pi, \phi_\tau, \psi$ and the natural maps
\[
{}^L\SO^0_{2n+1}(\bbC)\x\GL_r(\bbC)
=
{\rm Sp}_{2n}(\bbC)\x\GL_r(\bbC)
\hookto
\GL_{2n}(\bbC)\x\GL_r(\bbC)
\overset{\otimes}{\longto}\GL_{2nr}(\bbC)
\]
one can define the $L$-factor $L(s, \phi_\pi\ot\phi_\tau)$ and $\epsilon$-factor $\epsilon(s,\phi_\pi\ot\phi_\tau,\psi)$ as 
in \cite{Tate1979}. On the other side, one can define the (local) $Rankin$-$Selberg$ $L$-factor
 $L(s,\pi\x\tau)$ as a "g.c.d." of poles of local Rankin-Selberg integrals (\cite{Soudry1993}) for $good$ $sections$
 (cf. \S\ref{SS:RS local factor}). Granted the $L$- and $\gamma$-factors, the $Rankin$-$Selberg$ $\epsilon$-factors 
 $\epsilon(s,\pi\x\tau,\psi)$ are defined. We indicate that these analytic local factors can be defined even when $\tau$ is 
reducible.

\subsubsection{Induced of Langlands' type}\label{SSS:induced of Langlands' type}
As pointed out, we also consider reducible representations. To describe them, let $\tau$ be a smooth representation of 
$\GL_2(F)$. We denote by $\b{\tau}$ the unique non-zero irreducible quotient of $\tau$ (if exists). 
Certainly, $\tau\cong\b{\tau}$ when $\tau$ is irreducible. On the other hand, if $\tau$ is 
$induced$ $of$ $Langlands'$ $type$, namely, $\tau$ is a normalized induced representation of $\GL_2(F)$
inducing from characters (not necessary unitary) $\chi_1$ and $\chi_2$ of $F^\x$ with $|\chi_1\chi_2^{-1}(\varpi)|\le 1$, 
then $\b{\tau}$ is also defined. Here $\varpi$ is a fixed uniformizer of $F$.  Moreover, $\tau$ is reducible precise when  
$\chi_1\chi_2^{-1}=|\cdot|_F$, where $|\cdot|_F$ is the absolute value on $F$ normalized so that $|\varpi|_F=q^{-1}$ with
$q$ the size of the residue field of $F$. In this case, $\b{\tau}\cong\chi$ is one-dimensional with $\chi$ the character
of $F^\x$ given by $\chi=\chi_1|\cdot|_F^{-\frac{1}{2}}$, and we regard it as a character of $\GL_2(F)$ via composing 
with the determinant map. We indicate that every irreducible induced representations of $\GL_2(F)$ can be arranged into
induced of Langland's type, and every representation that is induced of Langlands' type admits a unique Whittaker model
(\cite{JacquetShalika1983}).\\

With these knowledge, our first result can be stated as follows.

\begin{thm}\label{T:main}
Let $\pi$ be an irreducible smooth generic representation of $\SO_5(F)$ and $\tau$ be a smooth 
generic representation of $\GL_2(F)$ that is irreducible or induced of Langlands' type.  Then we have
%\begin{equation}\label{E:RS=Gal r=2}
\[
L(s,\pi\x\tau)=L(s,\phi_\pi\ot\phi_{\b{\tau}})
\quad
and
\quad
\epsilon(s,\pi\x\tau,\psi)=\epsilon(s,\phi_\pi\ot\phi_{\b{\tau}},\psi).
\]
%\end{equation}
\end{thm}

\begin{Remark}
When $n\le 2$ and $r=1$, similar identities can be deduced from the results in the literature. Indeed, when 
$n=1$, $\SO_3\cong{\rm PGL}_2$ and the Rankin-Selberg integrals reduce to the zeta integrals studied by 
Jacquet-Langlands in \cite{JLbook}. In this case, the assertion is well-known.
When $n=2$, on the other hand, $\SO_5\cong{\rm PGSp}_4$ and the integrals 
become the zeta integrals (for ${\rm GSp}_4$) considered by Novodvorsky in \cite{Novodvorsky1979}. 
In this case, the assertion follows from the results in \cite{Takloo-Bighash2000}, \cite{Soudry2000}, 
\cite{JiangSoudry2004}, \cite{RobertsSchmidt2007}, \cite{Kaplan2015} and \cite{Tran2019}.
%We shall provide a proof for this case ($n=2$ and $r=1$) as we can't find a proper reference.
\end{Remark}

\subsubsection{Local factors via Novodvorsky's local integrals}
As noticed by Soudry in \cite[Section 0]{Soudry1993}, when $n=r=2$, the Rankin-Selberg integrals are essentially equal
to the (local) zeta integrals for ${\rm GSp}_4\x\GL_2$ introduced by Novodvorsky in \cite{Novodvorsky1979}.
These zeta integrals were subsequently studied by Soudry in \cite{Soudry1984} (see also \cite{LPSZ2021} and
\cite{Loeffler}). In particular, one can define the $L$-factors $L^{\rm Nov}(s,\pi\x\tau)$
and the $\epsilon$-factors $\epsilon^{\rm Nov}(s,\pi\x\tau,\psi)$ by using these zeta integrals (cf. \S\ref{SSS:zeta int}). 
Here $\pi$ (resp. $\tau$) is a smooth generic representation of ${\rm GSp}_4(F)$ (resp. $\GL_2(F)$). 
In this note, we explicate the relation between these two integrals and obtain the following result.

\begin{thm}\label{T:main'}
Let $\pi$ be an irreducible smooth generic representation of ${\rm GSp}_4(F)$ and $\tau$ be a smooth 
generic representation of $\GL_2(F)$ that is irreducible or induced of Langlands' type. Suppose that the central character
of $\pi$ is a square. Then we have
\[
L^{\rm Nov}(s,\pi\x\tau)=L(s,\phi_\pi\ot\phi_{\b{\tau}})
\quad
and
\quad
\epsilon^{\rm Nov}(s,\pi\x\tau,\psi)=\epsilon(s,\phi_\pi\ot\phi_{\b{\tau}},\psi).
\]
Here $\phi_\pi$ is the $L$-parameter of $\pi$ under the local Langlands correspondence for ${\rm GSp}_4$ 
established in \cite{GanTakeda2011}.
\end{thm}

\begin{Remark}
When $\pi$ is a theta lift from a generic representation of split ${\rm GSO}_4$, same identities was obtained by Soudry in  
\cite{Soudry1984} under a mild assumption. On the other hand, when $\tau$ is non-supercuspidal, the equality between 
$L$-factors was proved in recent results of Loeffler-Pilloni-Skinner-Zerbes, and Loeffler in \cite{LPSZ2021} and 
\cite{Loeffler}, respectively. We point out that in these results, there are no restriction on the central character.
Here we supply some complementary results under the assumption on the central character.
\end{Remark}

\subsection{Idea of the proof}
In an analogous setting, Kaplan proved, under an assumption
\footnote{This assumption can be lifted by results in a recent preprint \cite{Luo2021}}
on intertwining operators, that the $L$- and $\epsilon$-factors 
attached to $tempered$ irreducible smooth generic representations of (quasi)-split $\SO_{2n}(F)$ and $\GL_r(F)$ defined
by (LS) and (RS) are the same (\cite{Kaplan2013}). As pointed out in op. cit., the technique and results readily adapted to 
our settings, due to the similar nature and technical closeness of the constructions. In this note; however, we take a
different approach
\footnote{Actually, our original approach follows \cite{Kaplan2013}, but we can only obtain partial results.}
which takes the advantage of the accidental isomorphisms ${\rm PGSp}_4\cong\SO_5$ and 
$\bbG_m\backslash{\rm G}({\rm SL}_2\x{\rm SL}_2)\cong\SO_4$. These allow us to compare two local integrals and 
 then to establish the following key identity
 \[
 L^{\rm Nov}(s,\pi\x\tau)=L(s,\pi\x\tau).
 \] 
Once this identity is established, we can transfer the relevant results from one to another. In particular, by results in 
\cite{LPSZ2021} and \cite{Loeffler}, our task is to prove \thmref{T:main} when $\tau$ is supercuspidal. This is settled in
this note.

\subsection{Contents of the note}
We conclude this introduction with a brief description of the contents of this note. In \S\ref{S:RS int and local factor}, we will 
review the local Rankin-Selberg integrals for $\SO_{2n+1}\x\GL_r$ with $1\le r\le n$ and defined the associated local
factors. For this, we first introduce the good sections in \S\ref{SS:gd section}.  Then the integrals and the local
factors will be defined in \S\ref{SS:local factor}.  In \S\ref{SS:Laurent series}, we interpret the Rankin-Selberg integrals
as formal Laurent series following the idea of Jacquet-Shalika-Piatetski-Shapiro and Kaplan. This will be used in the proof
later. In \S\ref{SS:acc isom}, we recall the accidental isomorphisms and fit them into a commutative diagram, which is 
crucial when comparing two local integrals. Novodvorsky's local integrals and the associated local factors will be 
introduced in \S\ref{SS:zeta int}, and the comparison between Rankin-Selberg and Novodvorsky's local integrals will be 
given in \S\ref{SS:comparison}. Finally, the proof of \thmref{T:main} (when $\tau$ is supercuspidal) will appear in 
\S\ref{S:proof of main}.

\subsection{Conventions}
Let $G$ be an $\ell$-group in the sense of \cite[Section 1]{BZ1976}. 
In this note, by a representation of $G$ we mean a 
smooth representation with coefficients in $\bbC$. Let $\pi$ be a representation of $G$.  We usually denote 
by $\cV_\pi$ the underlying (abstract) $\bbC$-linear space and by $\omega_\pi$ the central character (if exists). 
When $G=\GL_r(F)$ or ${\rm GSp}_4(F)$ and $\mu$ is a character of $F^\x$, that is, a one-dimensional 
representation of $F^\x$, we often regard $\mu$ as a character of $G$ via composing with the determinant map or 
the similitude character accordingly. Then $\pi\ot\mu$ will be a representation of $G$ with $\cV_{\pi\ot\mu}=\cV_\pi$ 
and the action given by $\mu(g)\pi(g)v$ for $g\in G$ and $v\in\cV_\pi$.
If $\mu$ is a character of $F^\x$, we denote by $L(s,\mu)$, $\epsilon(s,\mu,\psi)$ and 
$\gamma(s,\mu,\psi)$ the $L$-, $\epsilon$- and $\gamma$-factor attached to $\mu$ and $\psi$ that appeared in the 
Tate thesis  (\cite[Section 3.1]{Bump1998}). 

\section{Rankin-Selberg Integrals and Local Factors}\label{S:RS int and local factor}
In this section, we review the local Rankin-Selberg integrals for $\SO_{2n+1}\x\GL_r$ with $1\le r\le n$ developed by 
Soudry in \cite{Soudry1993} (see also \cite{Kaplan2015}) and define the associated local factors. 

\subsection{Special orthogonal groups and their subgroups}\label{SS:SO}

\subsubsection{Special orthogonal groups}\label{SSS:special orthogonal group}
Define an element $\jmath_m\in\GL_m(F)$ inductively by 
\begin{equation}\label{E:J_m}
\jmath_1=1\quad\text{and}\quad \jmath_m=\pMX{0}{1}{\jmath_{m-1}}{0}.
\end{equation}
Put $S_m=\jmath_m$ if $m$ is even, and
\[
S_m
=
\begin{pmatrix}
&&\jmath_n\\
&2&\\
\jmath_n
\end{pmatrix}
\] 
if $m=2n+1$ is odd. Let $\SO_m(F)\subset{\rm SL}_m(F)$ be the special orthogonal group defined by 
\[
\SO_m(F)
=
\stt{g\in{\rm SL}_m(F)\mid {}^tgS_mg=S_m}
\]
where ${}^tg$ stands for the transpose of $g$. We may regard $\SO_m$ as an algebraic group defined over $\frak{o}$, 
the valuation ring of $F$.
%and denote $G_n=\SO_{2n+1}$ and $H_r=\SO_{2r}$ with $n,r\ge 1$.
Throughout this note, we assume $1\le r\le n$. Then there is a natural embedding 
$\varrho:\SO_{2r}(F)\hookto \SO_{2n+1}(F)$ 
given by 
\begin{equation}\label{E:embedding}
\SO_{2r}(F)\ni\pMX{a}{b}{c}{d}\longmapsto
\begin{pmatrix}
a&&b\\
&I_{2(n-r)+1}\\
c&&d
\end{pmatrix}\in \SO_{2n+1}(F)
\end{equation}
with $a,b,c,d\in{\rm Mat}_{r\x r}(F)$.

\subsubsection{Subgroups and non-degenerated characters}
Let $U_n\subset \SO_{2n+1}$, $V_r\subset\SO_{2r}$ and $Z_r\subset\GL_r$ be the maximal unipotent subgroups 
whose $F$-rational points consist upper unitriangular matrices. Define the non-degenerate characters 
$\psi_{U_n}: U_n(F)\to\bbC^\x$ and $\b{\psi}_{Z_r}:Z_r(F)\to\bbC^\x$ by 
\[
\psi_{U_n}(u)=\psi(u_{12}+u_{23}+\cdots+u_{n-1,n}+ 2^{-1}u_{n,n+1})
\]
and
\[
\b{\psi}_{Z_r}(z)
=
\b{\psi}(z_{12}+z_{23}+\cdots+z_{r-1, r})
\]
for $u=(u_{ij})\in U_n(F)$ and $z=(z_{ij})\in Z_r(F)$,  where $\b{\psi}=\psi^{-1}$.
Let $Q_r\subset \SO_{2r}$ be the Siegel parabolic subgroup with the Levi decomposition $M_r\ltimes N_r$ with
\begin{equation}\label{E:levi of Q}
M_r(F)
=
\stt{m_r(a)=\pMX{a}{}{}{a^*}\mid a\in\GL_r(F)}\cong\GL_r(F)
\end{equation}
and 
\begin{equation}\label{E:ur of Q}
N_r(F)
=
\stt{n_r(b)=\pMX{I_r}{b}{}{I_r}\mid\text{$b\in\Mat_{r\x r}(F)$ with $b=-\jmath_r {}^t b \jmath_r$}}.
\end{equation}
Here for $a\in\GL_r(F)$, we denote $a^*=\jmath_r{}^ta^{-1}\jmath_r$.

\subsection{Induced representations}\label{SS:ind rep}
We introduce normalized induced representations that enter the definitions of local Rankin-Selberg integrals. 
Let $\tau$ be a representation of $\GL_r(F)$, $s_0\in\bbC$ and denote 
$\tau_{s_0}=\tau\ot|\cdot|_F^{s_0-\frac{1}{2}}$.

\subsubsection{Induced representations $\rho_{\tau,s_0}$}
%Denote $\tau_{s_0}$ to be the representation of $\GL_r(F)$ acting on $\cV_\tau$ 
%with the action $\tau_{s_0}(a)=\tau(a)|\det(a)|_F^{s_0-\frac{1}{2}}$. 
By the decomposition $Q_r=M_r\ltimes N_r$, we can extend $\tau_{s_0}$ to a representation of $Q_r(F)$ and obtain a 
normalized induced representation 
\begin{equation}\label{E:induced rep}
\rho_{\tau, s_0}={\rm Ind}_{Q_r(F)}^{\SO_{2r}(F)}(\tau_{s_0})
\end{equation}
of $\SO_{2r}(F)$. Its underlying space $I(\tau,s_0)$ consists of smooth functions 
$\xi_{s_0}: \SO_{2r}(F)\to\cV_\tau$ satisfying 
\begin{equation}\label{E:xi_s}
\xi_{s_0}(m_r(a)nh)=\delta_{Q_r}^{\frac{1}{2}}(m_r(a))\tau_{s_0}(a)\xi_{s_0}(h)
\end{equation}
for $a\in\GL_r(F)$, $n\in N_r(F)$ and $h\in \SO_{2r}(F)$, where $\delta_{Q_r}$ is the modulus 
function of $Q_r(F)$. We remind that $\delta_{Q_r}(m_r(a))=|\det(a)|_F^{r-1}$.
%The group $H_r(F)$ acts on $I(\tau,s_0)$ by the right translation $\rho$.

\subsubsection{Induced representations $\t{\rho}_{\tau^*,1-s_0}$}
There are another induced representations that we need to introduce. First let $\t{Q}_r\subset\SO_{2r}$ be the maximal 
parabolic subgroup such that $\t{Q}_r(F)=\delta_r^{-r}Q_r(F)\delta^r_r$ where 
\footnote{The $\delta_r$ here is slightly different from that defined in \cite[Section 4.1]{YCheng}.}
\begin{equation}\label{E:delta}
\delta_r
=
\begin{pmatrix}
I_{r-1}&&\\
&\jmath_2&\\
&&I_{r-1}
\end{pmatrix}\in{\rm O}_{2r}(F)\setminus\SO_{2r}(F)
\end{equation}
(so that $\t{Q}_r=Q_r$ if $r$ is even) with ${\rm O}_{2r}(F)$ the split orthogonal group defined by $\jmath_{2r}$. 
The parabolic subgroup $\t{Q}_r$ has the Levi decomposition $\t{Q}_r=\t{M}_r\ltimes \t{N}_r$ with the Levi part 
$\t{M}_r(F)=\delta_r^{-r}M_r(F)\delta^r_r$ and the unipotent radical $\t{N}_r(F)=\delta_r^{-r}N_r(F)\delta^r_r$ 
(see \cite[Chapter 9]{Soudry1993} for more details).
Next, let us define $\tau^*$ to be the representation of $\GL_r(F)$ acting on $\cV_\tau$ with the action 
$\tau^*(a)=\tau(a^*)$ for $a\in\GL_r(F)$. Note that when $\tau$ is irreducible,  $\tau^*$ is isomorphic to the dual 
$\tau^\vee$. Again, let $s_0$ be a complex number.  Then by extending $\tau^*_{1-{s_0}}$ to the representation of 
$\t{Q}_r(F)$ on $\cV_\tau$ with the action
\[
\tau^*_{1-{s_0}}(\t{m}_r(a)\t{n})
=
\tau(a^*)|\det(a)|_F^{\frac{1}{2}-{s_0}}
\]
where $\t{m}_r(a)=\delta_r^{-r}m_r(a)\delta_r^r\in\t{M}_r(F)$ and 
$\t{n}\in\t{N}_r(F)$, we obtain another normalized induced representation 
\[
\t{\rho}_{\tau^*,1-{s_0}}
=
{\rm Ind}_{\t{Q}_r(F)}^{\SO_{2r}(F)}(\tau^*_{1-{s_0}})
\]
of $\SO_{2r}(F)$. Its underlying space $\t{I}(\tau^*,1-s_0)$ consists of the functions 
$\t{\xi}_{1-s_0}:\SO_{2r}(F)\to\cV_\tau$ satisfying the rule similar to that of \eqref{E:xi_s}.

\begin{Remark}\label{R:isom btw vs}
Observe that when $r$ is even, we have $\t{\rho}_{\tau^*,1-s_0}=\rho_{\tau^*,1-s_0}$ and 
$\t{I}(\tau^*,1-s_0)=I(\tau^*,1-s_0)$. In general, given $\t{\xi}_{1-s_0}\in \t{I}(\tau^*,1-s_0)$, we define 
$\t{\xi}^*_{1-s_0}(h)=\t{\xi}_{1-s_0}(\delta_r^{-r}h\delta_r^{r})$ for $h\in\SO_{2r}(F)$. 
Then $\t{\xi}^*_{1-s_0}\in I(\tau^*,1-s_0)$ and the map $\t{\xi}_{1-s_0}\mapsto\t{\xi}^*_{1-s_0}$ gives rise to an 
isomorphism between linear spaces $\t{I}(\tau^*,1-s_0)$ and $I(\tau^*,1-s_0)$.
\end{Remark}

\subsection{Intertwining operators and good sections}\label{SS:gd section}
As pointed out in the introduction, the definitions of the Rankin-Selberg $L$- and $\epsilon$-factors involve 
the so called $good$ $sections$. The notion of good sections was introduced by Piatetski-Shapiro and Rallis in 
\cite{PSR1987}, in order to define the local factors of triple product $L$-functions. This idea is adopted in 
\cite{Ikeda1992}, \cite{HKS1996}, \cite{Kaplan2013}, \cite{Yamana2014} and also in this note. 
In the followings, we use $s_0$ (resp. $s$) to denote a complex number (resp. variable) for the sake of clearness.

\subsubsection{Assumptions on $\tau$}\label{SSS:assump}
From now on, we always assume that $\tau$ is a subrepresentation of a representation (of $\GL_r(F)$) parabolically 
induced from an irreducible supercuspidal representation (of a Levi part), and the space 
${\rm Hom}_{Z_r(F)}(\tau,\bar{\psi}_{Z_r})=\bbC\,\Lambda_{\tau,\bar{\psi}_{Z_r}}$ is one-dimensional. 
In particular, $\tau$ admits a central character $\omega_\tau$, and all irreducible generic representations of 
$\GL_r(F)$ are included. We need these assumptions to ensure the intertwining operators that will be introduced in the 
next subsection enjoy certain properties (cf. Remark \ref{R:assump on tau}).
Observe that $\tau^*$ satisfies the same assumptions and is isomorphic to the admissible dual of $\tau$ when 
$\tau$ is irreducible. Moreover, the one-dimensional space ${\rm Hom}_{Z_r(F)}(\tau^*,\b{\psi}_{Z_r})$ is generated by 
$\Lambda_{\tau^*,\b{\psi}_{Z_r}}=\Lambda_{\tau,\b{\psi}_{Z_r}}\circ\tau(d_r)$ with
\begin{equation}\label{E:d_r}
d_r
=
\begin{pmatrix}
1&&&\\
&-1&&\\
&&\ddots\\
&&&(-1)^{r-1}
\end{pmatrix}
\in\GL_r(F).
\end{equation}

\subsubsection{Standard and holomorphic sections}\label{SSS:std and hol sec}
A function $\xi_s(h):\bbC\x\SO_{2r}(F)\to\cV_\tau$ is called a $section$ of the induced representation $\rho_{\tau,s_0}$ 
if the function $h\mapsto \xi_{s_0}(h)$ belongs to $I(\tau,s_0)$ for every $s_0\in\bbC$. 
A section $\xi_s(h)$ is called $standard$ if for every $k\in\SO_{2r}(\frak{o})$, the 
function $s\mapsto\xi_s(k)$ is a constant function. We denote by $I^{{\rm std}}(\tau,s)$ the space of standard sections.
Given a complex number $s_0$ and an element $\xi^\circ_{s_0}\in I(\tau,s_0)$, it is easy to see that 
there is a unique $\xi_s\in I^{{\rm std}}(\tau,s)$ so that $\xi_{s_0}=\xi^\circ_{s_0}$.
The space $I^{{\rm hol}}(\tau,s)$ of $holomorphic$ $sections$ is defined to be 
$I^{{\rm hol}}(\tau,s)=\bbC[q^{-s},q^s]\ot_\bbC I^{{\rm std}}(\tau,s)$.
Then by a similar argument in \cite[Claim 2.1]{Kaplan2013}, one checks that $I^{{\rm std}}(\tau,s)$ is preserved by the 
right translation of $h\in\SO_{2r}(F)$. The space of holomorphic sections 
$\t{I}^{{\rm hol}}(\tau^*,s)$ of $\t{\rho}_{\tau^*,s_0}$ can be defined in the similar way.

\subsubsection{Intertwining operators}\label{SSS:int for SO}
To define good sections, we first introduce the intertwining operator $M(\tau,s)$ and its normalization 
$M^\dagger_\psi(\tau,s)$ between the induced representations defined in \S\ref{SS:ind rep}.   Denote 
$\omega_r\in\SO_{2r}(F)$ by
\begin{equation}\label{E:omega_r}
\omega_r
=
\begin{cases}
\begin{pmatrix}&I_r\\I_r&\end{pmatrix}\quad&\text{if $r$ is even},\\
\begin{pmatrix}&I_r\\I_r&\end{pmatrix}\begin{pmatrix}&&1\\&I_{2r-2}&\\1&&\end{pmatrix}\quad&\text{if $r$ is odd}.
\end{cases}
\end{equation}
Then for $\Re(s_0)\gg 0$, the intertwining operator $M(\tau,s_0): I(\tau,s_0)\to\t{I}(\tau^*,1-s_0)$ is 
defined by the following convergent integral
\begin{equation}\label{E:int op}
M(\tau,s_0)\xi_{s_0}(h)
=
\int_{\t{N}_r(F)}\xi_{s_0}(\omega_r^{-1}\t{n}h)d\t{n}.
\end{equation}
For the choice of the Haar measure on $\t{N}_r(F)$, see \cite[Page 398]{Kaplan2015}. In particular, when $r=2$ so 
that $\t{N}_2(F)=N_2(F)\cong F$, $d\t{n}$ is simply the Haar measure on $F$ that is self-dual with respect to $\psi$. 
To define $M(\tau,s_0)\xi_{s_0}(h)$ for arbitrary $s_0$, one has to apply the meromorphic continuation. 
More precisely, given $\xi^\circ_{s_0}\in I(\tau,s_0)$, let $\xi_s\in I^{{\rm std}}(\tau,s)$ be the standard section such that 
$\xi_{s_0}=\xi^\circ_{s_0}$. We define $M(\tau,s)\xi_s(h)$ by the same integral \eqref{E:int op}. Then in addition to 
absolutely convergent for $\Re(s)\gg 0$, the integral also admits meromorphic continuation to whole complex plane. 
Now $M(\tau,s_0)\xi^\circ_{s_0}(h)$ is understood as $M(\tau,s)\xi_s(h)|_{s=s_0}$. One can define the intertwining 
operator $\t{M}(\tau^*, s_0)$ from $\t{I}(\tau^*, s_0)$ to $I(\tau, 1-s_0)$ in a similar fashion.\\

The normalized intertwining operator $M_\psi^\dagger(\tau,s)$ can be defined as follows. 
We first apply our second assumption on $\tau$ to define Shahidi's local coefficient 
\footnote{If $\tau$ is irreducible, then by \cite[Theorem 3.5]{Shahidi1990} and a recent result 
of Cogdell-Shahidi-Tsai in \cite{CogdellShahidiTsai2017}, the local coefficient $\gamma(s,\tau,\Wedge^2,\psi)$ and the 
$\gamma$-factor $\gamma(s,\phi_\tau,\Wedge^2,\psi)$ are equal up to a unit in $\bbC[q^{-s}, q^s]$.}
$\gamma(s,\tau,\Wedge^2,\psi)$ as in \cite[Page 521]{Soudry2000} (see also \cite[Pages 398-399]{Kaplan2015}).
Then $\gamma(s,\tau,\Wedge^2,\psi)$ lies in $\bbC(q^{-s})$ by our first assumption on $\tau$ together with
\cite[Proposition 2.1]{CasselmanShalika1980} and the following rationality results
%(cf. \cite[Pages 321-323]{Shahidi1981})
\begin{equation}\label{E:upper bound for int op}
P_0(q^{-2s})M(\tau,s)\xi_s\in\t{I}^{{\rm hol}}(\tau^*,1-s)
\quad\text{and}\quad 
P^*_0(q^{-2s})\t{M}(\tau^*,s)\t{\xi}_s\in I^{{\rm hol}}(\tau,1-s)
\end{equation}
for all $\xi_{s}\in I^{{\rm std}}(\tau,s)$ and $\t{\xi}_s\in\t{I}^{\rm std}(\tau^*,s)$, where $P_0(X)$ and $P^*_0(X)$ are some
non-zero polynomials. Now the normalized intertwining operator $M^\dagger_\psi(\tau,s)$ is defined to be
\begin{equation}\label{E:normalization}
M_\psi^\dagger(\tau, s)
=
\gamma(2s-1,\tau, {\Wedge}^2, \psi)M(\tau,s).
\end{equation}
One can define $\t{M}^\dagger(\tau^*,s)$ in a similar way. Then we have 
\begin{equation}\label{E:comp of int op}
M_\psi^\dagger(\tau,s)\circ \t{M}_\psi^\dagger(\tau^*,1-s)=1
\quad\text{and}\quad
\t{M}_{\psi}^\dagger(\tau^*,1-s)\circ M_{\psi}^\dagger(\tau,s)=1.
\end{equation}

\begin{Remark}\label{R:assump on tau}
There are two places that we need the assumptions on $\tau$ made in \S\ref{SSS:assump}. The first is the rationality
property of the intertwining operators in \eqref{E:upper bound for int op}. The second is the identities 
in \eqref{E:comp of int op}. When $\tau$ is irreducible, these follow from well-known results of Shahidi in 
\cite[Section 2,3]{Shahidi1981}. One of the key step to get these results is to realize $\tau$ as a 
subrepresentation of a representation that is parabolically induced from an irreducible supercuspidal representation.
When $\tau$ is irreducible, this follows from Jacquet's subrepresentation theorem. Here we make it as an assumption 
and then apply (variants of) the proofs in \cite{Shahidi1981} to obtain \eqref{E:upper bound for int op} and 
\eqref{E:comp of int op}.
\end{Remark}

\subsubsection{Good sections}
We now define good sections for $\rho_{\tau,s_0}$. By \eqref{E:upper bound for int op} and the fact that
$\gamma(s,\tau^*,\Wedge^2,\psi)\in\bbC(q^{-s})$, we have
$\t{M}_{\psi}^\dagger(\tau^*,1-s)\t{I}^{{\rm hol}}(\tau^*,1-s)\subset\bbC(q^{-s})\ot_\bbC I^{{\rm hol}}(\tau,s)$.
Thus we can define the space $I^{{\rm gd}}(\tau,s)$ of $good$ $sections$
%\footnote{Our definition is slightly different from that of \cite{Kaplan2013}.}
of $\rho_{\tau,s_0}$ to be the $\bbC[q^{-s}, q^s]$-submodule of $\bbC(q^{-s})\ot_\bbC I^{{\rm hol}}(\tau,s)$ 
generated by 
\begin{equation}\label{E:good section}
I^{{\rm hol}}(\tau,s)\cup\t{M}_{\psi}^\dagger(\tau^*,1-s)\t{I}^{{\rm hol}}(\tau^*,1-s).
\end{equation}
Similarly, we can define the space $\t{I}^{{\rm gd}}(\tau^*,1-s)$ of good sections of $\t{\rho}_{\tau^*,1-s_0}$. 
These definitions are independent of $\psi$ by \cite[(5.6)]{Kaplan2015}. Also by \eqref{E:comp of int op}, we have
\begin{equation}\label{E:close under M} 
M_{\psi}^\dagger(\tau,s)I^{{\rm gd}}(\tau,s)=\t{I}^{{\rm gd}}(\tau^*,1-s)
\quad\text{and}\quad
\t{M}^\dagger_\psi(\tau^*,1-s)\t{I}^{\rm gd}(\tau^*,1-s)=I^{\rm gd}(\tau,s).
\end{equation}
We point out that these identities are crucial in defining the Rankin-Selberg $L$- and $\epsilon$-factors attached to 
$\pi\x\tau$. Especially we need them to guarantee the (Rankin-Selberg) $\epsilon$-factors are units in $\bbC[q^{-s},q^s]$. 
For more reasons of using good sections instead of using merely holomorphic sections, we recommend the readers 
to consult \cite[pages 589-590]{Kaplan2013}.

\subsection{Integrals and local factors}\label{SS:local factor}
Let $\pi$ be an irreducible generic representation of $\SO_{2n+1}(F)$ and fix a non-zero element 
$\Lambda_{\pi,\psi_{U_n}}\in {\rm Hom}_{U_n(F)}(\pi,\psi_{U_n})$. Let $\tau$ be 
a representation of $\GL_r(F)$ as in \S\ref{SSS:assump}.
%Recall that $\SO_{2r}(F)$ can be regarded as a subgroup of $\SO_{2n+1}(F)$ via $\varrho$ (cf.\eqref{E:embedding}).

\subsubsection{Rankin-Selberg integrals}\label{SSS:RS int}
To define the local Rankin-Selberg integrals for $\pi\x\tau$, let us put
\[
\bar{X}_{n,r}(F)
=
\stt{
\begin{pmatrix}I_r&&&&\\x&I_{n-r}&&\\&&1\\&&&I_{n-r}&\\&&&x'&I_r\end{pmatrix}
\mid x\in{\rm Mat}_{(n-r)\x r}(F),\,\,x'=-\jmath_r{}^tx\jmath_{n-r}}
\subset\SO_{2n+1}(F).
\]
Then the integrals are given by 
\begin{equation}\label{E:RS integral in general}
\Psi(v\ot\xi_s)
=
\int_{V_r(F)\backslash\SO_{2r}(F)}\int_{\b{X}_{n,r}(F)}
W_v(\varrho(\b{x}h))f_{\xi_s}(h)d\b{x}dh
\end{equation}
where $W_v(g)=\Lambda_{\pi,\psi_{U_n}}(\pi(g)v)$ is the Whittaker function associated to $v\in\cV_\pi$ for 
$g\in\SO_{2n+1}(F)$ and $f_{\xi_s}(h)=\Lambda_{\tau,\bar{\psi}_{Z_r}}(\xi_s(h))$
for $\xi_s\in I^{{\rm std}}(\tau,s)$ and $h\in\SO_{2r}(F)$ (this is NOT the Whittaker function associated to $\xi_s$).
On the other hand, by using the normalized intertwining operator $M_\psi^\dagger(\tau,s)$, one obtains another 
integral $\t{\Psi}(v\ot\xi_s)$ attached to $v$ and $\xi_s$. More precisely, let
\[
\delta_{n,r}
=
\begin{pmatrix}
I_{r-1}&&\\&&&1\\&&-I_{2(n-r)+1}\\&1\\&&&&I_{r-1}
\end{pmatrix}\in\SO_{2n+1}(F).
\]
Then we define
\begin{equation}\label{E:dual RS int}
\t{\Psi}(v\ot\xi_s)
=
\int_{V_r(F)\backslash\SO_{2r}(F)}\int_{\b{X}_{n,r}(F)}
W_v(\varrho(\b{x}h)\delta^r_{n,r})f^*_{M^\dagger_\psi(\tau,s)\xi_s}(\delta_r^{-r}h\delta^r_r)d\b{x}dh
\end{equation}
where $f^*_{M^\dagger_\psi(\tau,s)\xi_s}(h)=\Lambda_{\tau^*,\b{\psi}_{Z_r}}(M^\dagger_\psi(\tau,s)\xi_s(h))$.\\

The integrals \eqref{E:RS integral in general} and \eqref{E:dual RS int}, which are originally absolute convergence
in some half-planes that depend only upon (the classes of) $\pi$ and $\tau$, have meromorphic 
continuations to whole complex plane, and give rise to rational functions in $q^{-s}$. Furthermore, 
there exist $v\in\cV_\pi$ and $\xi_s\in I^{{\rm std}}(\tau,s)$ such that $\Psi(v\ot\xi_s)=1$ according to
\cite[Section 6]{Soudry1993}.  Consequently
\footnote{By combining with the multiplicity one property of certain Hom spaces (\cite[Chapter 8]{Soudry1993}).}
there exists a non-zero rational function $\gamma(s,\pi\x\tau,\psi)\in\bbC(q^{-s})$, depending only on $\psi$ and 
(the classes of) $\pi$ and $\tau$ such that 
\begin{equation}\label{E:FE}
\t{\Psi}(v\ot\xi_s)=\gamma(s,\pi\x\tau,\psi)\Psi(v\ot\xi_s)
\end{equation}
for all $v\in\cV_\pi$ and $\xi_s\in \bbC(q^{-s})\ot_{\bbC} I^{{\rm std}}(\tau,s)$. Concerning the compatibility for the 
$\gamma$-factors defined by different methods, we have the following result.

\begin{thm}[\cite{JiangSoudry2004}, \cite{Soudry2000}, \cite{Kaplan2015}]\label{T:RS=Gal gamma}
Let $\pi$ be an irreducible generic representation of $\SO_{2n+1}(F)$ and $\tau$ be an irreducible generic representation 
of $\GL_r(F)$. Then we have $\omega_\tau(-1)^n\gamma(s,\pi\x\tau,\psi)=\gamma(s,\phi_\pi\ot\phi_\tau,\psi)$.
\end{thm}

The following lemma, which is standard, is needed in formulating the Rankin-Selberg $L$- and $\epsilon$-factors

\begin{lm}\label{L:upper bound of RS integral}
Let $v\in\cV_\pi$ and $\xi_s\in I^{{\rm gd}}(\tau,s)$. 
Then we have $\Psi(v\ot\xi_s)\in\bbC(q^{-s})$ as a meromorphic function. 
Moreover, there exists a non-zero polynomial $P(X)\in\bbC[X]$ depending at most upon $\psi$ and (the classes of)
$\pi$ and $\tau$ such that $P(q^{-s})\Psi(v\ot\xi_s)\in\bbC[q^{-s},q^s]$.
\end{lm}

\begin{proof}
If $\xi_s\in I^{{\rm std}}(\tau,s)$, then the assertions follows from results in \cite[Section 4.3]{Soudry1993}, namely, 
$\Psi(v\ot\xi_s)$ gives an element in $\bbC(q^{-s})$ and there exists $P_1(X)\in \bbC[X]$ such that 
$P_1(q^{-s})\Psi(v\ot\xi_s)$ is contained in $\bbC[q^{-s}, q^s]$. 
Since $I^{{\rm hol}}(\tau,s)=\bbC[q^{-s}, q^s]\ot_\bbC I^{\rm std}(\tau,s)$, same assertions hold when 
$\xi_s\in I^{{\rm hol}}(\tau,s)$. To complete the proof, we may assume $\xi_s=M^\dagger_\psi(\tau^*,1-s)\t{\xi}_{1-s}$ for 
some $\t{\xi}_{1-s}\in\t{I}^{{\rm std}}(\tau^*,1-s)$ by \eqref{E:good section}. Then by \eqref{E:upper bound for int op} and  
the fact that $\gamma(s,\tau^*,\Wedge^2,\psi)$ is a non-zero element in $\bbC(q^{-s})$, there exists a non-zero 
$P_2(X)\in\bbC[X]$ so that 
\[
P_2(q^{-s})\xi_s=P_2(q^{-s})M^\dagger(\tau^*,1-s)\t{\xi}_{\tau^*,1-s}\in I^{{\rm hol}}(\tau,s).
\] 
Now the proof follows at once if we set $P(X)=P_1(X)P_2(X)$.
\end{proof}

\subsubsection{Rankin-Selberg local factors}\label{SS:RS local factor}
We are ready to define the (local) Rankin-Selberg $L$- and $\epsilon$-factors attached to $\pi\x\tau$.
Let $I_{\pi\x\tau}(s)\subset\bbC(q^{-s})$ be the $\bbC[q^{-s}, q^s]$-submodule spanned by the integrals 
$\Psi(v\ot\xi_s)$ for $v\in\cV_\pi$ and $\xi_s\in I^{{\rm gd}}(\tau,s)$.  The submodule $I_{\pi\x\tau}(s)$ is independent
of $\psi$ by \cite[Proposition 10.2]{Soudry1993} (see also \cite[Section 5.1]{Kaplan2015}). Moreover,
\lmref{L:upper bound of RS integral} implies that $I_{\pi\x\tau}(s)$ is in fact a fractional ideal. Now since 
$1\in\ I_{\pi\x\tau}(s)$ by \cite[Chapter 6]{Soudry1993}, the ideal $I_{\pi\x\tau}(s)$ admits a generator 
$P_{\pi\x\tau}(q^{-s})^{-1}$ for some $P_{\pi\x\tau}(X)\in\bbC[X]$ with $P_{\pi\x\tau}(0)=1$. We then define
\[
L(s,\pi\x\tau)=P_{\pi\x\tau}(q^{-s})^{-1}
\quad\text{and}\quad
\epsilon(s,\pi\x\tau,\psi)=\omega_\tau(-1)^n\gamma(s,\pi\x\tau,\psi)\frac{L(s,\pi\x\tau)}{L(1-s,\pi\x\tau^*)}.
\]
Then $I_{\pi\x\tau}(s)=L(s,\pi\x\tau)\bbC[q^{-s}, q^s]$ according to the definition. On the other hand, a standard argument 
implies that $\epsilon(s,\pi\x\tau,\psi)$ is a unit in $\bbC[q^{-s}, q^s]$ (\cite[Claim 3.2]{Kaplan2013}).

\begin{Remark}
The definition of $L(s,\pi\x\tau)$ seems to depend on the integrals when $r$ is odd, that is, one can define the 
$L$-factor by using either $\Psi(v\ot\xi_s)$ or $\t{\Psi}(v\ot\xi_s)$, while the induced representations involved are different.
However, it does not since in \eqref{E:dual RS int}, the section is twisted by $\delta_r$, and the isomorphism in 
Remark \eqref{R:isom btw vs} extends to an isomorphism between the spaces of relevant standard sections due to
the fact $\delta_r^{-1}\SO_{2r}(\frak{o})\delta_r=\SO_{2r}(\frak{o})$.
\end{Remark}

\subsection{Integrals and formal Laurent series}\label{SS:Laurent series}
In this subsection, we associate to each integral $\Psi(v\ot\xi_s)$ (with $v\in\cV_\pi$ and $\xi_s\in I^{{\rm hol}}(\tau,s)$)
a $formal$ $Laurent$ $series$ $\Psi_{v\ot\xi_s}(X)$ such that $\Psi_{v\ot\xi_s}(q^{-{s_0}})=\Psi(v\ot\xi_{s_0})$ for 
all $s_0$ with $\Re(s_0)\gg 0$. These formal Laurent series will be used in our proofs. 
Let $\bbC\llbracket X, X^{-1}\rrbracket$ be the space of the formal sums $f(X)=\sum_{m\in\bbZ}c_mX^m$ 
with $c_m\in\bbC$. Following \cite{JPSS1983}, elements in $\bbC\llbracket X, X^{-1}\rrbracket$ are called formal 
Laurent series. Note that $\bbC\llbracket X, X^{-1}\rrbracket$ can be regard (in a natural way) as a 
$\bbC[X,X^{-1}]$-module with torsion.
The idea of interpreting integrals as formal Laurent series first appeared in \cite{JPSS1983}, which allows them to 
prove the multiplicativity of their $\gamma$-factors, and as a byproduct, the upper bounds of poles of their integrals. 
In \cite[section 8]{Kaplan2013}, Kaplan adopted the same idea to bound the poles for his integrals. He also provided
some background definitions and results which explicate the connection between integrals and series in 
\cite[Section 7]{Kaplan2013}. In this note, however, we don't need such a generality.\\

Let $v\in\cV_\pi$ and $\xi_s\in I^{\rm hol}(\tau,s)$. Then the integral $\Psi(v\ot\xi_s)$ can be formally written as 
\begin{equation}\label{E:exp RS int}
\Psi(v\ot\xi_s)
=
\sum_{m\in\bbZ}q^{-ms}\Psi^m(v\ot\xi_s)
\end{equation}
with
\begin{equation}\label{E:coeff of FLS}
\Psi^m(v\ot\xi_s)
=
\int_{Z_r(F)\backslash\GL^m_r(F)}\int_{\SO_{2r}(\frak{o})}\int_{\b{X}_{n,r}(F)}
W_v(\b{x}m_r(a)k)W'_{\xi_s(k)}(a)|\det(a)|_F^{-\frac{r}{2}}d\b{x}dkd^{\x}a
\end{equation}
by the Iwasawa decomposition $\SO_{2r}(F)=Q_r(F)\SO_{2r}(\frak{o})$ and the identification
$V_r(F)\backslash Q_r(F)\cong Z_r(F)\backslash\GL_r(F)$. Here we denote
%\begin{equation}\label{E:Whittaker for GL}
\[
W'_{\xi_s(k)}(a)
=
\Lambda_{\tau,\b{\psi}_{Z_r}}(\tau(a)\xi_s(k))
\]
%\end{equation}
and
\[
\GL^m_r(F)
=
\stt{a\in\GL_r(F)\mid |\det(a)|_F=q^{-m}}
\]
for $m\in\bbZ$.
Suppose that $\xi_s$ is a standard section. Then the integrals $\Psi^m(v\ot\xi_s)$ are actually independent of $s$. 
We show that they converge absolutely
\footnote{For this only, one can also apply the Funibi's theorem.}
and vanishing for all $m\ll 0$. Indeed, it follows from the proof of 
\cite[Proposition 4.2]{Soudry1993} that the integral $\Psi^m(v\ot\xi_s)$ is bounded by a finite sum of integrals of the form
%\begin{equation}\label{E:estimate RS int}
\[
\int f(\a_1,\a_1,\cdots, \a_r)\chi_1(\a_1)\chi_2(\a_2)\cdots\chi_r(\a_r)|\a_1\a_2^2\cdots\a_r^r|_F^{\Re(s_0)}
d^\x\a_1 d^\x\a_2\cdots d^\x\a_r
\]
%\end{equation}
with the domain of integration 
\[
\stt{(\a_1,\a_2,\ldots, \a_r)\in F^{\x,r}\mid  |\a_1\a_2^2\cdots\a_r^r|_F=q^{-m}}.
\]
Here $f\ge 0$ is a Bruhat-Schwartz function on $F^r$ and $\chi_1,\chi_2,\ldots,\chi_r\in\frak{X}_{\pi,\tau}$ where 
$\frak{X}_{\pi,\tau}$ is a finite set of positive characters of $F^\x$ that depends only on (the classes of) $\pi$ and $\tau$. 
The assertions then follow immediately.\\

Now we define the associated formal Laurent series by 
\[
\Psi_{v\ot\xi_s}(X)
=
\sum_{m\in\bbZ}X^m\Psi^m(v\ot\xi_s)
\]
for $v\in\cV_\pi$ and $\xi_s\in I^{\rm std}(\tau,s)$. In general, if $\xi_s$ is a holomorphic section, then 
$\xi_s=\sum_j P_j(q^{-s})\xi^j_s$ is a finite sum of standard sections $\xi_s^j$ with coefficients 
$P_j(q^{-s})\in\bbC[q^{-s},q^s]$. We then define 
\[
\Psi_{v\ot\xi_s}(X)
=
\sum_j P_j(X)\Psi_{v\ot\xi^j_s}(X).
\]
Evidently, these formal Laurent series have only finitely many negative terms and 
$\Psi(v\ot\xi_s)(q^{-s_0})=\Psi(v\ot\xi_{s_0})$ for all $s_0$ in some right half-plane plane.

\section{Novodvorsky v.s. Soudry}\label{S:N vs S}
Let $\pi$ be an irreducible generic representation of $\SO_5(F)$ and $\tau$ be a generic representation of 
$\GL_2(F)$ that is irreducible or induced of Langlands' type (cf. \S\ref{SSS:induced of Langlands' type}). 
In particular, $\tau$ satisfies the assumptions in \S\ref{SSS:assump}. On the other hand, via the accidental isomorphism 
between ${\rm PGSp}_4$ and $\SO_5$, $\pi$ can also be regarded as an irreducible smooth generic 
representation of ${\rm GSp}_4(F)$ with trivial central character. In \cite{Novodvorsky1979}, Novodvorsky suggested 
a constriction of $L$- and $\epsilon$-functions attached to generic automorphic representations of 
${\rm GSp}_4\x\GL_2$ over global fields, whose corresponding local theory was later studied by Soudry in 
\cite{Soudry1984}. To $v\in\cV_\pi$, $u\in\cV_\tau$ and a Bruhat-Schwartz function $\phi$ on $F^2$, one can attach 
the Novodvorsky's local integral $Z^{\rm Nov}(s,v\ot u,\phi)$. Then as indicated in 
\cite[Section 0]{Soudry1993}, the Rankin-Selberg integrals introduced in \S\ref{SSS:RS int} essentially equal to  
Novodvorsky's integrals when $n=r=2$. The goal of this section is to explicate their relation, and as a consequence, 
to show that the $L$-factors defined by these two integrals coincide.

\subsection{Accidental isomorphisms}\label{SS:acc isom}
In this subsection only, let $\bbF$ be an arbitrary field with the characteristic different from $2$. We describe the 
accidental isomorphisms $\b{\vartheta}:{\rm PGSp}_4\cong\SO_5$ and 
$\b{\vartheta}':\bbG_m\backslash{\rm G}({\rm SL}_2\x{\rm SL}_2)\cong\SO_4$, and then fit them and the embedding 
$\varrho$ in \eqref{E:embedding} (with $n=r=2$) into a commutative diagram. 

\subsubsection{The accidental isomorphism $\vartheta$}%\label{SSS:vartheta}
Let $(W,\langle,\rangle)$ be the $4$-dimensional symplectic space over $\bbF$ and 
\[
{\rm GSp}(W)
=
\stt{g\in\GL(W)\mid \langle gw, gw'\rangle=\nu(g)\langle w, w'\rangle\,\,\text{for all $w,w'\in W$}}
\]
be the similitude symplectic group with $\nu:{\rm GSp}(W)\to\bbF^\x$ the similitude character. 
We fix an ordered basis $\stt{w_1, w_2, w^*_2, w^*_1}$ of $W$ so that the associated Gram matrix is given by
\[
J=\pMX{}{\jmath_2}{-\jmath_2}{}.
\]
Let $(\t{V}, (,))$ be the $6$-dimensional quadratic space over $\bbF$ with $\t{V}=\bigwedge^2 W$ and the symmetric 
bilinear form $(,)$ defined by 
\[
v_1\wedge v_2=(v_1,v_2)(w_1\wedge w_2\wedge w^*_1\wedge w^*_2)
\]
for $v_1,v_2\in\t{V}$. Let $\t{v}=w_1\wedge w^*_1+w_2\wedge w^*_2\in\t{V}$ and put 
$V=\stt{v\in\t{V}\mid (v,\t{v})=0}$. Denote also by $(,)$ the restriction of the symmetric bilinear form to $V$, so that 
$(V,(,))$ becomes a $5$-dimensional quadratic space over $\bbF$. Let $\t{\vartheta}:{\rm GSp}(W)\to\SO(\t{V})$ be the 
homomorphism defined by $\t{\vartheta}(g)=\nu(g)^{-1}\Wedge^2 g$ for $g\in {\rm GSp}(W)$. 
Then because $\t{\vartheta}(h)\t{v}=\t{v}$, the homomorphism induces an exact sequence 
\[
1\longto\bbF^\x\overset{\iota}{\longto}{\rm GSp}(W)\overset{\vartheta}{\longto}\SO(V)\longto 1
\]
where $\iota(a)=aI_{W}$ with $I_W:W\to W$ the identity map. Now $\b{\vartheta}$ is the one induced from this exact 
sequence. 

\subsubsection{The accidental isomorphism $\vartheta'$}\label{SSS:vartheta'}
Let $(V',(,)')$ be the $4$-dimensional quadratic space over $\bbF$ with $V'={\rm M}_2(\bbF)$ the space of two-by-two 
matrices with entries in $\bbF$, and the symmetric bilinear form $(,)'$ given by
\[
(v'_1,v'_2)
=
{\rm det}(v_1'+v_2')-{\rm det}(v_1')-{\rm det}(v_2')
\]
for $v_1', v_2'\in V'$. Here ${\rm det}$ stands for the determinant map from $V'$ onto $\bbF$.  Let 
\[
{\rm GSO}(V')
=
\stt{h\in\GL(V')\mid (hv,hv')'=\nu'(h)(v,v')'\,\,\text{for all $v,v'\in V'$}}
\]
be the similitude special orthogonal group with $\nu':{\rm GSO}(V')\to\bbF^\x$ the similitude character. Note that 
$\SO(V')=\ker(\nu')$. Define a surjective homomorphism $\t{\vartheta}':\GL_2(\bbF)\x\GL_2(\bbF)\longto{\rm GSO}(V')$ 
by $\t{\vartheta}'(a_1,a_1)v'=a_1v'a_2^{-1}$ for $a_1, a_2\in\GL_2(\bbF)$ and $v'\in V'$. Then we have an exact sequence
\[
1\longto\bbF^\x\overset{\iota'}{\longto}\GL_2(\bbF)\x\GL_2(\bbF)\overset{\t{\vartheta}'}{\longto}{\rm GSO}(V')\longto 1
\]
with $\iota'(\a)=(\a I_2, \a I_2)$ for $\a\in\bbF^\x$. Since 
$\nu'(\t{\vartheta}'(a_1,a_2))={\rm det}(a_1){\rm det}(a_2)^{-1}$, the restriction gives rise to an exact sequence
\[
1\longto\bbF^\x\overset{\iota'}{\longto}{\rm G}({\rm SL}_2\x{\rm SL}_2)(\bbF)
\overset{\vartheta'}{\longto}{\rm SO}(V')\longto 1
\]
where
\[
{\rm G}({\rm SL}_2\x{\rm SL}_2)(\bbF)
=
\stt{(a_1,a_2)\in\GL_2(\bbF)\x\GL_2(\bbF)\mid {\rm det}(a_1)={\rm det}(a_2)}.
\]
Now $\b{\vartheta}'$ is the one induced from this exact sequence. 

\subsubsection{A commutative diagram}\label{SSS:comm diag}
We fit $\vartheta$, $\vartheta'$ and $\varrho$ into a commutative diagram. For this, let 
\[
e_1=w_1\wedge w_2,\,\, e_2=w_1\wedge w^*_2,\,\ v_0=w_1\wedge w^*_1-w_2\wedge w_2^*,\,\,f_2=w_2\wedge w_1^*,
\,\,f_1=w_1^*\wedge w_2^*
\]
be an ordered basis of $V$ so that its Gram matrix is the matrix $S_2$ in \S\ref{SSS:special orthogonal group}.
On the other hand, let
\[
e_1'
=
\pMX{0}{1}{0}{0};
\quad
e_2'
=
\pMX{-1}{0}{0}{0};
\quad
f_2'
=
\pMX{0}{0}{0}{-1};
\quad
f_1'
=
\pMX{0}{0}{-1}{0}
\]
be an ordered basis of $V'$. Then the associated Gram matrix is $\jmath_4$. Using these bases, 
we can identify $\SO(V)$ with $\SO_5(\bbF)$ and ${\rm GSp}_4(W)$ with ${\rm GSp}_4(\bbF)$. Next, let 
$\varrho':{\rm G}({\rm SL}_2\x{\rm SL}_2)(\bbF)\longto{\rm GSp}_4(\bbF)$ be the embedding given by 
\begin{equation}\label{E:rho'}
\left(\pMX{\a}{\beta}{\gamma}{\delta}, \pMX{\a'}{\beta'}{\gamma'}{\delta'}\right)
\mapsto
\begin{pmatrix}
\a&&&\beta\\
&\a'&\beta'\\
&\gamma'&\delta'\\
\gamma&&&\delta
\end{pmatrix}.
\end{equation}
We then have the following commutative diagram
\begin{equation}\label{E:comm diag}
\begin{tikzcd}
&1\arrow{r}&\bbF^\x\arrow{d}{\parallel}\arrow{r}{\iota'}&{\rm G}({\rm SL}_2\x{\rm SL}_2)(\bbF)\arrow{d}{\varrho'}
\arrow{r}{\vartheta'}&\SO_4(\bbF)\arrow{d}{\varrho}\arrow{r}&1\\
&1\arrow{r}&\bbF^\x\arrow{r}{\iota}&{\rm GSp}_4(\bbF)\arrow{r}{\vartheta}&\SO_5(\bbF)\arrow{r}&1.
\end{tikzcd}
\end{equation}
This diagram will be used later when we compare Rankin-Selberg integrals with Novodvorsky' zeta integrals.

\subsection{Novodvorsky's zeta integrals and local factors}\label{SS:zeta int}
We introduce Novodvorsky's zeta integrals for generic representations of ${\rm GSp}_4\x\GL_2$ and then 
define the associated local factors in this subsection. Let $\pi$ be an irreducible generic representation of 
${\rm GSp}_4(F)$ (see \cite[Page 34]{RobertsSchmidt2007}) and $\tau$ be a generic representation of 
$\GL_2(F)$ that is irreducible or induced of Langlands' type and put $\omega=\omega_\pi\omega_\tau$. 
Given $v\in\cV_\pi$ (resp. $u\in\cV_\tau$), we let $W_v$ (resp. $W'_u$) be the associated Whittaker function with 
respect to $\psi$ (resp. $\b{\psi}$).\\

For convenience, let us introduce the following notations
\begin{equation}\label{E:matrix}
t(\a,\beta)=\pMX{\a}{}{}{\beta};
\quad
z(y)=\pMX{1}{y}{}{1};
\quad
w=\pMX{}{1}{1}{}
\end{equation}
for $\a,\beta\in F^\x$ and $y\in F$.

\subsubsection{Induced representations for $\GL_2$}
Let $B_2\subset\GL_2$ be the standard upper triangular Borel subgroup. For a pair $\ul{\chi}=(\chi_1,\chi_2)$ of 
characters of $F^\x$. Then $\ul{\chi}$ induces a character on $B_2(F)$ by 
$\ul{\chi}(t(\a,\beta)z(y))=\chi_1(\a)\chi_2(\delta)$. Let 
\[
\tau_{\ul{\chi}}={\rm Ind}_{B_2(F)}^{\GL_2(F)}(\ul{\chi})
\]
to a the normalized induced representation of $\GL_2(F)$. Its underlying space, which denote by $\cB(\ul{\chi})$, 
consists of smooth functions $f:\GL_2(F)\to\bbC$ satisfying $f(bg)=\ul{\chi}(b)\delta^{\frac{1}{2}}_{B_2}(b)f(g)$ for 
$b\in B_2(F)$ and $g\in \GL_2(F)$, where $\delta_{B_2}$ is the modulus function of $B_2$. More generally, for a complex
number $s_0$, we denote $\ul{\chi}_{s_0}=(\chi_1|\cdot|_F^{s_0-\frac{1}{2}}, \chi_2|\cdot|_F^{\frac{1}{2}-s_0})$.
Then we have the induced representation $\tau_{\ul{\chi}_{s_0}}$ with the underlying space 
$\cB(\ul{\chi},s_0)=\cB(\ul{\chi}_{s_0})$. The spaces $\cB^{\rm std}(\ul{\chi},s)$ and $\cB^{\rm hol}(\ul{\chi},s)$ of 
standard and holomorphic sections of $\tau_{\ul{\chi}_{s_0}}$ can be defined in similar ways as in 
\S\ref{SSS:std and hol sec}.\\

Let $\ul{\chi}^\iota=(\chi_2,\chi_1)$. Then $M(\ul{\chi},s_0):\cB(\ul{\chi},s_0)\longto\cB(\ul{\chi}^\iota,1-s_0)$  the 
intertwining map can be defined as follows. For $s_0$ with $\Re(s_0)\gg 0$ (depending only on $\ul{\chi}$), it's given by 
the convergent integral
\begin{equation}\label{E:int op GL}
M(\ul{\chi},s_0)f_{s_0}(a)
=
\int_F f_{s_0}(wz(y)a)dy.
\end{equation}
Here $dy$ is the Haar measure on $F$ that is self-dual with respect to $\psi$. For an arbitrary $s_0$, 
$M(\ul{\chi},s_0)f_{s_0}$ can be defined as in \S\ref{SSS:int for SO}.

\subsubsection{Godement sections}
Let $\cS(F^2)$ be the space of Bruhat-Schwartz functions on $F^2$. For $\phi\in\cS(F^2)$, we define the $Godement$
$section$
\[
f^{\phi}_s(a;\ul{\chi})
=
\chi_1({\rm det}(a))|{\rm det}(a)|^s\int_{F^\x}\phi((0,\a)a)\chi_1\chi_2^{-1}|\a|_F^{2s}d^\x\a
\]
for $a\in\GL_2(F)$. Here $d^\x\a$ is any Haar measure on $F^\x$. Let us review some of its properties.
First note that the integral is essentially a local Tate integral associated to the character $\chi_1\chi_2^{-1}$ 
(for every fixed $a$). Consequently, this integral converges absolutely in some right half-plane and admits a meromorphic 
continuation to whole complex plane. Moreover, as a meromorphic function, it becomes an element in $\bbC(q^{-s})$, and 
\[
L(2s,\chi_1\chi_2^{-1})^{-1}f^{\phi}_s(a;\ul{\chi})\in\bbC[q^{-s},q^s]
\]
for every $a\in\GL_2(F)$ and $\phi\in\cS(F^2)$. By changing variables, one sees immediately that (as a function
on $\GL_2(F)$)
\[
f^\phi_{s_0}(-;\ul{\chi})\in\cB(\ul{\chi},s_0)
\]
for every $s_0$ at which it is defined. Furthermore, if we define the (symplectic) Fourier transform $\h{\phi}$ of $\phi$ by 
\[
\h{\phi}(x,y)=\int_F\int_F \phi(x',y')\psi(x'y-xy')dx'dy'
\]
with $dx'$, $dy'$ the Haar measures on $F$ that is self-dual with respect to $\psi$, then as meromorphic functions
\begin{equation}\label{E:local coeff for GL_2}
\chi_1(-1)\gamma(2s-1,\chi_1\chi_2^{-1},\psi)M(\ul{\chi},s)f^\phi_s(a;\ul{\chi})=f^{\h{\phi}}_{1-s}(a;\ul{\chi}^\iota)
\end{equation}
by \cite[Section 4.B]{GelbartJacquet1979}. Finally, we have the following lemma, whose proof is borrowed from that of
\cite[Proposition 3.2]{JLbook}.

\begin{lm}\label{L:god sec}
For every $f_s\in\cB^{\rm std}(\ul{\chi},s)$, there exists $\phi\in\cS(F^2)$ such that $f^\phi_s(-,\ul{\chi})=f_s$. 
On the other hand, we have $L(2s,\chi_1\chi^{-1}_2)^{-1}f^\phi_s(-;\ul{\chi})\in\cB^{\rm hol}(\ul{\chi},s)$ 
for every $\phi\in\cS(F^2)$.
\end{lm}

\begin{proof}
Let $f_s\in\cB^{\rm std}(\ul{\chi},s)$ be given. We define $\phi(x,y)\in\cS(F^2)$ as follows. 
If $(x,y)\notin (0,1)\GL_2(\frak{o})$, then $\phi(x,y)=0$, while if $(x,y)=(0,1)a$ for some $a\in\GL_2(F)$, then we set 
$\phi(x,y)=c^{-1}\chi_1({\rm det}(a))^{-1}f_s(a)$, where $c={\rm vol}(\frak{o}^\x,d^\x\a)$. 
Note that $\phi(x,y)$ is independent of $s$ since $f_s$ is a standard section. 
Moreover, it's clear from the definition that $\phi\in\cS(F^2)$. We show that $f^\phi_s(a;\ul{\chi})$ is absolute convergence 
for all $a$ and $f^\phi_s(-;\ul{\chi})=f_s$. For this, it suffices to check that $f^\phi_s(a;\ul{\chi})=f_s(a)$ for 
$a\in\GL_2(\frak{o})$. But by the definition of $\phi$, we have
\begin{align*}
f^{\phi}_s(a;\ul{\chi})
&=
\chi_1({\rm det}(a))\int_{F^\x}\phi((0,\a)a)\chi_1\chi_2^{-1}|\a|_F^{2s}d^\x\a\\
&=
\chi_1({\rm det}(a))\int_{\frak{o}^\x}\phi((0,1)t(1,\a)a)\chi_1\chi_2^{-1}(\a)d^\x\a\\
&=
f_s(a).
\end{align*}
This proves the first assertion.\\

To prove the second assertion, we note that by definition, $L(2s,\chi_1\chi^{-1}_2)^{-1}f^\phi_s(-;\ul{\chi})$ is right 
$K_0$-invariant and $L(2s,\chi_1\chi^{-1}_2)^{-1}f^\phi_s(a;\ul{\chi})\in\bbC[q^{-s},q^s]$ for all $a\in\GL_2(F)$, where
$K_0\subset\GL_2(\frak{o})$ is some open compact subgroup that only depends on $\phi$. 
So our task is to show that if $f_s$ is a non-zero section of $\tau_{\ul{\chi}_{s_0}}$ that is right $K_0$-invariant for 
some open compact subgroup $K_0\subset\GL_2(\frak{o})$ that is independent of $s$, and 
$f_s(a)\in\bbC[q^{-s},q^s]$ for all $a\in\GL_2(F)$, then $f_s\in\cB^{\rm hol}(\ul{\chi},s)$. 
Let $k_1,\ldots, k_m$ be representatives of the double cosets $B_2(F)\backslash\GL_2(F)/K_0$. 
Let $r=m$ if $f_s(k_j)\ne 0$ for all $1\le j\le m$. On the other hand, if $f_s(k_j)=0$ for some $1\le j\le m$, then after 
reindex, we may assume $f_s(k_j)\neq 0$ for $1\le j\le r$ and $f_s(k_j)=0$ for $r<j\le m$. For each $1\le j\le r$, define 
$f^{(j)}_s(a)=\ul{\chi}_s(b)\delta_{B_2}(b)^{\frac{1}{2}}$ if $a=bk_jk$ for some $b\in B_2(F)$ and $k\in K_0$, and 
$f^{(j)}_s(a)=0$ if $a\notin B_2(F)k_jK_0$. Then $f^{(j)}_s\in\cB^{\rm std}(\ul{\chi},s)$ are well-defined and we have 
\[
f_s=f_s(k_1)f_s^{(1)}+\cdots+f_s(k_r)f^{(r)}_s\in\cB^{\rm hol}(\ul{\chi},s).
\]
This completes the proof.
\end{proof}

\subsubsection{Zeta integrals and the associated local factors}\label{SSS:zeta int}
Now we are in the position to introduce Novodvorsky's zeta integrals for $\pi\x\tau$ and define the associated local factors. 
For the ease of notation, let us denote $f^\phi_s(-;\mu)$ to be $f^\phi_s(-;(1,\mu))$ for every character $\mu$ of $F^\x$.
Using the embedding $\varrho'$ (cf. \eqref{E:rho'}), we form the integral
\[
Z(s,v\ot u,\phi)
=
\int_{F^\x Z_2^\square(F)\backslash{\rm G}({\rm SL}_2\x{\rm SL}_2)(F)}
W_v(\varrho'(a_1,a_2))f_s^\phi(a_1;\omega^{-1})W'_u(a_2)d^\x (a_1,a_2)
\]
for $v\in\cV_\pi$, $u\in\cV_\tau$ and $\phi\in\cS(F^2)$, where
\[
Z_2^\square(F)
=
\stt{(z(x),z(y))\in{\rm G}({\rm SL}_2\x{\rm SL}_2)(F)\mid x,y\in F}
\]
and $F^\x$ is viewed as a subgroup of ${\rm G}({\rm SL}_2\x{\rm SL}_2)(F)$ via the embedding $\iota'$.
Moreover, we have put $W'_u(a)=\Lambda_{\tau,\b{\psi}_{Z_2}}(\tau(a)u)$ for $u\in\cV_\tau$ and $a\in\GL_2(F)$.
As expected, this integral converges absolutely in some right half-plane, admits meromorphic continuation to whole 
complex plane, and becomes an element in $\bbC(q^{-s})$. Moreover, if we put
\[
I^{\rm Nov}_{\pi\x\tau}(s)=\stt{Z(s,v\ot u,\phi)\mid v\in\cV_\pi, u\in\cV_\tau, \phi\in\cS(F^2)}\subset\bbC(q^{-s})
\]
then it forms a fractional ideal that contains $1$. As a result, the ideal $I^{\rm Nov}_{\pi\x\tau}(s)$ has the generator 
$P^{\rm Nov}_{\pi\x\tau}(q^{-s})^{-1}$ for some $P^{\rm Nov}_{\pi\x\tau}(X)\in\bbC[X]$ with $P^{\rm Nov}_{\pi\x\tau}(0)=1$.
We define
\[
L^{\rm Nov}(s,\pi\x\tau)=P^{\rm Nov}_{\pi\x\tau}(q^{-s})^{-1}
\]
to be the Novodvorsky's local $L$-factor attached to $\pi\x\tau$.\\

Next, let us describe the function equation and then define the Novodvorsky's $\epsilon$- and $\gamma$-factor attached 
to $\pi\x\tau$ and $\psi$. Given $v\in\cV_\pi$ and $u\in\cV_\tau$, we put $\t{W}_v(g)=W_v(g)\omega_\pi^{-1}(\nu'(g))$
and $\t{W}'_u(a)=W'_u(a)\omega_{\tau}^{-1}({\rm det(a)})$ for $g\in{\rm GSp}_4(F)$ and $a\in\GL_2(F)$. Then 
$\t{W}_v$ (resp. $\t{W}'_u$) defines an element in the Whittaker model of $\pi^\vee$ (resp. $\tau^*$) with respect to
$\psi$ (resp. $\b{\psi}$) with $\pi^\vee$ stands for the admissible dual of $\pi$. Note that if $\tau$ is induced of 
Langlands' type, then $\tau^*$ is also. Now we define
\[
\t{Z}(s,v\ot u,\phi)
=
\int_{F^\x Z^\square(F)\backslash{\rm G}({\rm SL}_2\x{\rm SL}_2)(F)}
\t{W}_v(\varrho'(a_1,a_2))f_s^\phi(a_1;\omega)\t{W}'_u(a_2)d^\x (a_1,a_2).
\]
Then the functional equation reads
\[
\t{Z}(1-s,v\ot u,\h{\phi})
=
\gamma^{\rm Nov}(s,\pi\x\tau,\psi)
Z(s,v\ot u,\phi)
\]
every $v$, $u$ and $\phi$, where $\gamma^{\rm Nov}(s,\pi\x\tau,\psi)\in\bbC(q^{-s})$ denotes the Novodvorsky's 
$\gamma$-factor. The Novodvorsky's $\epsilon$-factor is defined in a similar way, namely, 
\[
\epsilon^{\rm Nov}(s,\pi\x\tau,\psi)
=
\gamma^{\rm Nov}(s,\pi\x\tau,\psi)\frac{L^{\rm Nov}(s,\pi\x\tau)}{L^{\rm Nov}(1-s,\pi^\vee\x\tau^*)}.
\]
Applying the identity $\h{\h{\phi}}=\phi$ and the facts $(\pi^\vee)^\vee\cong\pi$, $(\tau^*)^*\cong\tau$, 
one deduces easily that $\epsilon^{\rm Nov}(s,\pi\x\tau,\psi)$ is a unit in $\bbC[q^{-s},q^s]$.\\

Now we record results in the literature on compatibility between the local factors defined by Novodvorsky's zeta integrals
and the associated Weil-Deligne representations. For ${\rm GSp}_4(F)$, we use the local Langlands correspondence 
established by Gan-Takeda in \cite{GanTakeda2011}. We remark here that when $\pi$ has trivial central character so that 
it can also be regarded as a representation of $\SO_5(F)$, the associated $L$-parameters
defined by \cite{JiangSoudry2004}, \cite{RobertsSchmidt2007} and \cite{GanTakeda2011} are all the same.

\begin{thm}[\cite{Soudry1984}, \cite{GanTakeda2011}]\label{T:theta}
Let $\pi$ be an irreducible generic representation of ${\rm GSp}_4(F)$ and $\tau$ be an irreducible 
representation of $\GL_2(F)$. Suppose that $\phi_\pi=\phi_{\tau_1}\oplus\phi_{\tau_2}$ for some irreducible 
generic representations $\tau_1$ and $\tau_2$ of $\GL_2(F)$ with $\omega_{\tau_1}=\omega_{\tau_2}$, and if $\tau$ is 
supercuspidal, then $\tau$ is not isomorphic to any representations of the form $\tau_j^{\vee}\ot|\cdot|^{s_j}$ for some
$s_j\in\bbC$ with $j=1,2$. Then we have
\[
L^{\rm Nov}(s,\pi\x\tau)=L(s,\phi_\pi\ot\phi_\tau)
\quad\text{and}\quad
\epsilon^{\rm Nov}(s,\pi\x\tau,\psi)=\epsilon(s,\phi_\pi\ot\phi_\tau,\psi).
\]
\end{thm}

More recently, we have the following results on the compatibility between $L$-factors.

\begin{thm}[\cite{LPSZ2021}, \cite{Loeffler}]\label{T:non-sc}
Let $\pi$ be an irreducible generic representation of ${\rm GSp}_4(F)$ and $\tau$ be a generic representation of 
$\GL_2(F)$ that is irreducible or induced of Langlands' type. Suppose that $\tau$ is non-supercusipdal.
Then we have 
\[
L^{\rm Nov}(s,\pi\x\tau)=L(s,\phi_\pi\ot\phi_{\b{\tau}}).
\]
\end{thm}

\begin{Remark}
Strictly speaking, results in the literature deal with the case where $\tau$ is irreducible. However, same results carry over
to the case when $\tau$ is reducible. In particular, \cite[Theorem 2.1]{Soudry1984} and \cite[Theorem 8.9(i)]{LPSZ2021}
are applicable to the case when $\tau$ is induced of Langlands' type.
\end{Remark}

\subsection{A comparison between two integrals}\label{SS:comparison}
The aim of this subsection is to prove the following proposition, which is the core of this note.

\begin{prop}\label{P:key}
Let $\pi$ be an irreducible generic representation of $\SO_5(F)$ and $\tau$ be a generic representation of 
$\GL_2(F)$ that is irreducible or induced of Langlands' type. Then we have 
\[
L(s,\pi\x\tau)=L^{\rm Nov}(s,\pi\x\tau).
\]
Here we also regard $\pi$ as a representation of ${\rm GSp}_4(F)$ with trivial central character.
\end{prop}

To prove proposition \propref{P:key}, some preparations are needed. 

\subsubsection{Induced representations of ${\rm GSO}_4(F)$}
Using the basis $\stt{e_1', e_2', f_2', f_1'}$ in \S\ref{SSS:comm diag}, we have the following commutative diagram
\[
\begin{tikzcd}
&1\arrow{r}& F^\x\arrow{d}{\parallel}\arrow{r}{\iota'}&{\rm G}({\rm SL}_2\x{\rm SL}_2)(F)\arrow{d}
\arrow{r}{\vartheta'}&\SO_4(F)\arrow{d}\arrow{r}&1\\
&1\arrow{r}& F^\x\arrow{r}{\iota'}&\GL_2(F)\x\GL_2(F)\arrow{r}{\t{\vartheta}'}&{\rm GSO}_4(F)\arrow{r}&1
\end{tikzcd}
\]
where the second and the third vertical lines are natural inclusions. Now if $\tau_1$ and $\tau_2$ are 
representations of $\GL_2(F)$ with $\omega_{\tau_1}\omega_{\tau_2}=1$, then $\tau_1\bt\tau_2$ can be 
viewed as a representation of ${\rm GSO}_4(F)$ via the bottom exact sequence. On the other hand, let 
$Q_2'\subset{\rm GSO}_4$ be the Siegel parabolic subgroup whose intersection with $\SO_4$ is $Q_2$. 
It admits a Levi decomposition $Q'_2=M'_2\ltimes N_2$ with
\[
M_2'(F)
=
\stt{m_2(a,\beta)=\pMX{a}{}{}{\beta a^*}\mid a\in\GL_2(F)\,\,\text{and}\,\,\beta\in F^\x}\cong\GL_2(F)\x\GL_1(F). 
\] 
Given a representation $\tau$ of $\GL_2(F)$ and character $\mu$ of $F^\x$, we obtain a representation of 
$Q_2'(F)$ on $\cV_\tau$ with the action $(\tau\bt\mu)(m_2(a,\beta)n)u=\mu(\beta)\tau(a)u$ for $u\in\cV_\tau$
and $n\in N_2(F)$. Let 
\[
\rho'_{\tau\bt\mu}
=
{\rm Ind}_{Q'_2(F)}^{{\rm GSO}_4(F)}(\tau\bt\mu)
\]
be a normalized induced representation of ${\rm GSO}_4(F)$ whose underlying space $I'(\tau\bt\mu)$ consisting 
smooth functions $\xi':{\rm GSO}_4(F)\to\cV_\tau$ satisfying 
\[
\xi'(m_2(a,\beta)nh)=\delta_{Q'_2}^{\frac{1}{2}}(m_2(a,\beta))\mu(\beta)\tau(a)\xi'(h)
\]
for $n\in N_2(F)$ and $h\in{\rm GSO}_4(F)$, where $\delta_{Q'_2}$ is the modulus function of $Q'_2$ and is given by 
$\delta_{Q'_2}(m_2(a,\beta))=|\det(a)|_F|\beta|_F^{-1}$. More generally, for $s_0\in\bbC$, let $\rho'_{\tau\bt\mu,s_0}$ 
be the representation of ${\rm GSO}_4(F)$ inducing from the datum 
$(\tau\ot|\cdot|_F^{s_0-\frac{1}{2}}, \mu|\cdot|_F^{\frac{1}{2}-s_0})$. Its underlying space is also denoted by
$I'(\tau\bt\mu,s_0)$. The spaces $I'^{\rm std}(\tau\bt\mu,s)$ and $I'^{\rm hol}(\tau\bt\mu,s)$
of standard and holomorphic sections of $\rho'_{\tau\bt\mu,s_0}$ can be defined in similar ways as in
\S\ref{SSS:std and hol sec}.

\subsubsection{A key lemma}
The following lemma, whose proof is easy, plays an important role in the sequel.

\begin{lm}\label{L:key}
Let $\ul{\chi}=(\chi_1,\chi_2)$ be a pair of characters of $F^\x$ and $\tau$ be a representation of $\GL_2(F)$ with 
$\omega_\tau\chi_1\chi_2=1$. Then under the identification $(\GL_2(F)\x\GL_2(F))/F^\x\cong{\rm GSO}_4(F)$, we have
$\tau_{\ul{\chi}}\bt\tau\cong\rho'_{(\tau\ot\chi_1)\bt\chi_2}$.
\end{lm}

\begin{proof}
Let $f\in\cB(\ul{\chi})$ and $u\in\cV_\tau$. Define a function $\xi'_{f\ot u}:{\rm GSO}_4(F)\to\cV_\tau$ by  
\begin{equation}\label{E:xi'}
\xi'_{f\ot u}(\t{\vartheta}'(a_1,a_2))=f(a_1)\tau(a_2)u
\end{equation}
for $a_1, a_2\in\GL_2(F)$. Then $\xi'_{f\ot u}\in I'((\tau\ot\chi_1)\bt\chi_2)$ and the map $f\ot u\mapsto \xi'_{f\ot u}$ from 
$\cB(\ul{\chi})\ot\cV_\tau$ onto $I'((\tau\ot\chi_1)\bt\chi_2)$ meets the requirement.
\end{proof}

\lmref{L:key} has an application to compute the local coefficients $\gamma(s,\tau,\Wedge^2,\psi)$. For this we note 
that the restriction induces a surjection from $I'^{\rm std}(\tau\bt\mu, s)$ onto $I^{\rm std}(\tau,s)$.

\begin{cor}\label{C:LC}
Let $\tau$ be a generic representation of $\GL_2(F)$ that is irreducible or induced of Langlands' type. We have 
$\gamma(s,\tau,\Wedge^2,\psi)=\gamma(s,\omega_\tau,\psi)$.
\end{cor}

\begin{proof}
When $\tau$ is irreducible, the assertion would follow from the results in \cite{Shahidi1990} and 
\cite{CogdellShahidiTsai2017}, but since $\tau$ can be reducible, we shall provide a proof here.
Let $\xi_s\in I^{\rm std}(\tau,s)$. Then by definition of the local coefficient,  
\begin{equation}\label{E:LC}
\int_F \xi_s(\omega_2n_2(y))\psi(y)dy
=
\gamma(2s-1,\tau,\Wedge^2,\psi)
\int_F M(\tau,s)\xi_s(\omega_2n_2(y))\psi(y)dy
\end{equation}
where $dy$ is an arbitrary Haar measure on $F$, and we have slightly abuse the notation (cf. \eqref{E:ur of Q})
to denote 
\[
n_2(y)
=
\begin{pmatrix}
1&&-y&\\
&1&&y\\
&&1\\
&&&1
\end{pmatrix}\in\SO_4(F)
\]
for $y\in F$. To compute $\gamma(s,\tau,\Wedge^2,\psi)$, the idea is find a proper $\ul{\chi}=(\chi_1,\chi_2)$ and then 
use \lmref{L:key}. More precisely, let $\chi_1=1$, $\chi_2=\omega_\tau^{-1}$ and consider the induced representation 
$\rho'_{\tau\bt\omega^{-1}_\tau,s_0}$ of ${\rm GSO}_4(F)$. On one hand, we have 
\[
\tau_{\ul{\chi}_{s_0}}\bt\tau\cong\rho'_{\tau\bt\omega_\tau^{-1},s_0}
\]
by \lmref{L:key} and the isomorphism (cf. \eqref{E:xi'})
\[
\cB^{\rm std}(\ul{\chi},s)\ot_\bbC\cV_\tau\cong I'^{\rm std}(\tau\bt\omega_{\tau}^{-1},s).
\]
On the other hand, we have a surjection from $I'^{\rm std}(\tau\bt\omega_{\tau}^{-1},s)$ onto $I^{\rm std}(\tau,s)$ 
given by the restriction. In view of these, we may assume that $\xi_s$ in \eqref{E:LC} is a restriction of $\xi'_{f_s\ot u}$
 for some $f_s\in\cB^{\rm std}(\ul{\chi},s)$ and $u\in\cV_\tau$. Here $\xi'_{f_s\ot u}$ is the $\cV_\tau$-valued function on
${\rm GSO}_4(F)$ defined by \eqref{E:xi'} with $f$ replaced by $f_s$.\\

Now simple computations show that $\vartheta'(w,d_2)=\omega_2$ and $\vartheta'(z(y), I_2)=n_2(y)$ for $y\in F$.
We remind here that $d_2$ is given by \eqref{E:d_r}. From these, we get that 
\[
\xi_s(\omega_2n_2(y))
=
\xi'_{f_s\ot u}(\vartheta'(wz(y),d_2))
=
f_s(wz(y))\tau(d_2)u
\]
and
\begin{equation}\label{E:int op conn}
M(\tau,s)\xi_s(\vartheta'(a_1,a_2))
=
\int_F\xi'_{f_s\ot u}(\vartheta'(wz(y)a_1,d_2a_2))dy
=
M(\ul{\chi},s)f_s(a_1)\tau(d_2a_2)u
\end{equation}
for $(a_1,a_2)\in{\rm G}({\rm SL}_2\x{\rm SL}_2)(F)$ by \eqref{E:int op GL}. It follows that the LHS of 
\eqref{E:LC} becomes  
\[
\tau(d_2)u\int_F f_s(wz(y))\psi(y)dy
\]
while the RHS of \eqref{E:LC} can be written as
\[
\gamma(2s-1,\tau,\Wedge^2,\psi)
\tau(d_2)u\int_F M(\ul{\chi},s)f_s(wz(y))\psi(y)dy.
\]
Combining these with \cite[Proposition 4.5.9]{Bump1998}, the proof follows.
\end{proof}

\subsubsection{Proof of \propref{P:key}}
We first relate the Rankin-Selberg integrals and Novodvorsky's zeta integrals. The idea of which is similar to that of 
the proof of \corref{C:LC}. Let $\xi_s\in I^{\rm std}(\tau,s)$ and $v\in\cV_\pi$. Recall that the Rankin-Selberg integral 
$\Psi(v\ot\xi_s)$ is defined by 
\[
\Psi(v\ot\xi_s)
=
\int_{V_2(F)\backslash\SO_4(F)} W_v(\varrho(h))f_{\xi_s}(h)dh
\]
with $f_{\xi_s}(h)=\Lambda_{\tau,\b{\psi}_{Z_2}}(\xi_s(h))$ for $h\in\SO_4(F)$.\\

Let $\ul{\chi}=(1,\omega_\tau^{-1})$ be a pair of characters of $F^\x$. Then we have the isomorphism 
(cf. \eqref{E:xi'})
\[
\cB^{\rm std}(\ul{\chi},s)\ot_\bbC\cV_\tau\cong I'^{\rm std}(\tau\bt\omega_{\tau}^{-1},s)
\]
between representations of ${\rm GSO}_4(F)$ and the surjection from $I'^{\rm std}(\tau\bt\omega_{\tau}^{-1},s)$ onto 
$I^{\rm std}(\tau,s)$ given by the restriction. To connect the integrals, assume that $\xi_s$ is a restriction of $\xi'_{f_s\ot u}$
 for some $f_s\in\cB^{\rm std}(\ul{\chi},s)$ and $u\in\cV_\tau$. Here $\xi'_{f_s\ot u}$ is the $\cV_\tau$-valued function on
${\rm GSO}_4(F)$ defined by \eqref{E:xi'} with $f$ replaced by $f_s$. By \lmref{L:god sec}, we may further assume that 
$f_s=f_s^\phi(-;\omega_\tau^{-1})$ for some $\phi\in\cS(F^2)$. Recall that we also denote $f^\phi_s(-;\mu)$ to be 
$f^\phi_s(-;(1,\mu))$ for every character $\mu$ of $F^\x$. Then for $(a_1,a_2)\in{\rm G}({\rm SL}_2\x{\rm SL}_2)(F)$, we 
have
\[
f_{\xi_s}(\vartheta'(a_1,a_2))
=
\Lambda_{\tau,\b{\psi}_{Z_2}}(\xi'_{f_s\ot u}(\vartheta'(a_1,a_2)))
=
f_s(a_1)W'_u(a_2)
=
f^\phi_s(a_1;\omega_\tau^{-1})W'_u(a_2).
\]
Now since $\vartheta'$ maps $Z_2^\square(F)$ onto $V_2(F)$, one sees that
\begin{equation}\label{E:RS=N}
\Psi(v\ot\xi_s)
=
\int_{F^\x Z_2^\square(F)\backslash{\rm G}({\rm SL}_2\x{\rm SL}_2)(F)}
W_v(\vartheta(\varrho'(a_1,a_2)))f_s^\phi(a_1;\omega_\tau^{-1})W'_u(a_2)d^\x (a_1,a_2)
=
Z(s,v\ot u,\phi)
\end{equation}
by the commutative diagram \eqref{E:comm diag}.\\ 

We also need to connect $\t{\Psi}(v\ot\xi_s)$ with $\t{Z}(1-s,v\ot u,\h{\phi})$. For this we first recall that 
\[
\t{\Psi}(v\ot\xi_s)
=
\int_{V_2(F)\backslash\SO_4(F)} W_v(\varrho(h))f^*_{M^\dagger_\psi(\tau,s)\xi_s}(h)dh
\]
with $f^*_{\xi_s}(h)=\Lambda_{\tau^*,\b{\psi}_{Z_2}}(M^\dagger_\psi(\tau,s)\xi_s(h))$ for $h\in\SO_4(F)$.
Let  $(a_1,a_2)\in{\rm G}({\rm SL}_2\x{\rm SL}_2)(F)$. Then by \eqref{E:int op conn}, we have
\[
f^*_{M^\dagger_\psi(\tau,s)\xi_s}(\vartheta'(a_1,a_2))
=
M^\dagger_\psi(\ul{\chi},s)f_s^\phi(a_1;\omega_\tau^{-1})W'_u(a_2)
\]
with $M^\dagger_\psi(\ul{\chi},s):=\gamma(2s-1,\tau,\Wedge^2,\psi)M(\ul{\chi},s)$. Now since 
\[
M^\dagger_\psi(\ul{\chi},s)f^\phi_s(a_1;\omega_\tau^{-1})=f^{\h{\phi}}_{1-s}(a_1;(\omega_\tau^{-1},1))
\]
by \eqref{E:local coeff for GL_2} and \corref{C:LC}, and
\[
f^{\h{\phi}}_{1-s}(a_1;(\omega_\tau^{-1},1))W'_u(a_2)=\t{W}'_u(a_2)f^{\h{\phi}}_{1-s}(a_2;\omega_\tau)
\]
due to $\det(a_1)=\det(a_2)$, we get that 
\begin{equation}\label{E:dual RS=N}
\t{\Psi}(v\ot\xi_s)=\t{Z}(1-s,v\ot u,\h{\phi})
\end{equation}
by a similar argument.\\

The identities just derived have two consequences. The first consequence is the following identities
\begin{equation}\label{E:RS=N=Gal gamma}
\gamma^{\rm Nov}(s,\pi\x\tau,\psi)
=
\gamma(s,\pi\x\tau,\psi)
=
\gamma(s,\phi_\pi\ot\phi_{\b{\tau}},\psi).
\end{equation}
Indeed, if $\tau$ is irreducible, then it comes from \eqref{E:RS=N}, \eqref{E:dual RS=N} and \thmref{T:RS=Gal gamma}.
Suppose that $\tau$ is reducible so that $\tau=\tau_{\ul{\mu}}$ for a pair of characters $\ul{\mu}=(\mu_1,\mu_2)$ of 
$F^\x$ with $\mu_1\mu_2^{-1}=|\cdot|_F$. Note that $\phi_{\b{\tau}}=\phi_{\mu_1}\oplus\phi_{\mu_2}$.
In this case, the first identity remains valid, so it suffices to establish the 
second identity. But this follows at once from the multiplicativity of the Rankin-Selberg $\gamma$-factors
%$\gamma(s,\pi\x\tau,\psi)=\gamma(s,\pi\x\mu_1,\psi)\gamma(s,\pi\x\mu_2,\psi)$ 
proved in \cite{Soudry2000} and \thmref{T:RS=Gal gamma}. As another consequence, we show that
\begin{equation}\label{E:ideals}
I_{\pi\x\tau}(s)\subseteq I_{\pi\x\tau}^{\rm Nov}(s)\subseteq L(2s,\omega_\tau)I_{\pi\x\tau}(s).
\end{equation}
To verify the first containment, it is enough to show that $\Psi(v\ot\xi_s)\in I^{\rm Nov}_{\pi\x\tau}(s)$ for $v\in\cV_\pi$
and $\xi_s\in I^{\rm std}(\tau,s)$ or $\xi_s=M^\dagger_\psi(\tau^*,1-s)\xi^*_{1-s}$ for some 
$\xi^*_{1-s}\in I^{\rm std}(\tau^*,1-s)$. If $\xi_s\in I^{\rm std}(\tau,s)$, then the assertion follows from \eqref{E:RS=N}.
Suppose that $\xi_s=M^\dagger_\psi(\tau^*,\psi)\xi^*_{1-s}$ for some $\xi^*_{1-s}\in I^{\rm std}(\tau^*,1-s)$.
We may assume that $\xi^*_{1-s}$ is the restriction of $\xi'_{f^*_{1-s}\ot u}$ with 
$f^*_{1-s}=f^\phi_{1-s}(-;\omega_\tau)\in\cB^{\rm std}((1,\omega_\tau),1-s)$ for some $u\in\cV_\tau$ and 
$\phi\in\cS(F^2)$. Then
\[
\Psi(v\ot\xi_s)
=
Z(s,v\ot u', \h{\phi})
\]
where $u'$ is the unique element in $\cV_\tau$ such that $W'_{u'}(a)=W'_u(d_2a^*)\omega^{-1}_\tau(a)$ for 
$a\in\GL_2(F)$. This can be derived in a similar way as we obtained \eqref{E:dual RS=N}. This verifies the first 
containment, while the second containment follows from \lmref{L:god sec} and \eqref{E:RS=N}.\\

Now we are ready to prove the identity between the $L$-factors. First note that if $\tau=\tau_0\ot|\cdot|_F^{s_0}$ for some 
representation $\tau_0$ of $\GL_2(F)$ and $s_0\in\bbC$, then 
\begin{equation}\label{E:shift}
L^{\rm Nov}(s,\pi\x\tau)=L^{\rm Nov}(s+s_0,\pi\x\tau_0)
\quad\text{and}\quad 
L(s,\pi\x\tau)=L(s+s_0,\pi\x\tau_0).
\end{equation} 
So we may assume without loss of generality that $\omega_\tau$ is unitary.
The containments in \eqref{E:ideals} imply 
\[
L(s,\pi\x\tau)=L^{\rm Nov}(s,\pi\x\tau)P(q^{-s})
\]
where $P(q^{-s})\in\bbC[q^{-s},q^s]$ is a factor of $L(2s,\omega_\tau)^{-1}$. Replacing
\footnote{Strictly speaking, we shall also replace $\pi$ with $\pi^\vee$ in $L^{\rm Nov}(s,\pi\x\tau)$, but since 
$\pi$ has trivial central character (as a representation of ${\rm GSp}_4(F)$), it is actually self-dual by 
\cite[Proposition 2.3]{Takloo-Bighash2000}.}
$\tau$ with $\tau^*$ and $s$ with $1-s$, we also have
\[
L(1-s,\pi\x\tau^*)=L^{\rm Nov}(1-s,\pi\x\tau^*)P^*(q^{1-s})
\]
with $P^*(q^{1-s})\in\bbC[q^{-s},q^s]$ a factor of $L(2-2s,\omega_\tau^{-1})^{-1}$. Since 
$\gamma^{\rm Nov}(s,\pi\x\tau,\psi)=\gamma(s,\pi\x\tau,\psi)$ and the $L$-factors $L(2s,\omega_\tau)$ and 
$L(2-2s,\omega^{-1}_\tau)$ have no common poles, we conclude that $L^{\rm Nov}(s,\pi\x\tau)=L(s,\pi\x\tau)$ as desired.
\qed

\subsubsection{Two corollaries}
\propref{P:key} has two corollaries. 

\begin{cor}\label{C:non-sc}
\thmref{T:main} holds when $\tau$ is non-supercuspidal.
\end{cor}

\begin{proof}
This follows immediately from \thmref{T:RS=Gal gamma}, \thmref{T:non-sc} and \propref{P:key}.
\end{proof}

\begin{cor}\label{C:non-sc'}
\thmref{T:main'} holds when $\tau$ is non-supercuspidal.
\end{cor}

\begin{proof}
To prove this corollary, we simply note that 
\[
L^{\rm Nov}(s,(\pi\ot\mu)\x\tau)=L^{\rm Nov}(s,\pi\x(\tau\ot\mu))
\]
for every character $\mu$ of $F^\x$. Now if $\omega_\pi=\mu^2$ for some character $\mu$, then $\pi_0:=\pi\ot\mu^{-1}$
has trivial central character, and hence 
\[
L^{\rm Nov}(s,\pi\x\tau)
=
L^{\rm Nov}(s,(\pi_0\ot\mu)\x\tau)
=
L^{\rm Nov}(s,\pi_0\x(\tau\ot\mu))
=
L(s,\pi_0\x(\tau\ot\mu))
=
L(s,\phi_{\pi_0}\ot\phi_{\ol{\tau\ot\mu}})
\]
by \propref{P:key} and \corref{C:non-sc}. On the other hand, since  
\[
\phi_{\pi_0}\ot\phi_{\ol{\tau\ot\mu}}
\cong 
\phi_{\pi_0}\ot(\phi_{\b{\tau}}\ot\phi_\mu)
\cong
(\phi_{\pi_0}\ot\phi_\mu)\ot\phi_{\b{\tau}}
\cong
\phi_{\pi}\ot\phi_{\b{\tau}}
\]
as $8$-dimensional representations by \cite[Main Theorem (iv)]{GanTakeda2011}, we get that 
$L^{\rm Nov}(s,\pi\x\tau)=L(s,\phi_\pi\ot\phi_{\b{\tau}})$.
Finally, the result for the $\epsilon$-factors follows from this and \eqref{E:RS=N=Gal gamma}.
\end{proof}

\section{Proof of the main results}\label{S:proof of main}
We prove \thmref{T:main} and \thmref{T:main'} in this section. By \corref{C:non-sc} and \corref{C:non-sc'}, it remains to
prove these Theorems when $\tau$ is supercuspidal. So let us assume that $\tau$ is supercuspidal from now on. 
In view of \eqref{E:shift}, we may further assume that $\omega_\tau$ is unitary.  Moreover, once \thmref{T:main} 
is verified, one can apply the similar argument in \corref{C:non-sc'} to prove \thmref{T:main'}. Therefore, our focus will be 
on prove \thmref{T:main}. By \thmref{T:RS=Gal gamma}, it suffices to establish the identity between the $L$-factors. 
To establish such an identity, we often use the following principle:
\textit{If the $L$-factors $L(s,\pi\x\tau)$ and $L(1-s,\pi\x\tau^\vee)$ have no common poles and the same holds for the 
$L$-factors $L(s,\phi_\pi\ot\phi_{\tau})$ and $L(1-s,\phi_\pi\ot\phi_{\tau^\vee})$, then 
$L(s,\pi\x\tau)=L(s,\phi_\pi\ot\phi_{\tau})$.} 
This is again a simple consequence of \thmref{T:RS=Gal gamma}.\\ 
%Note that we write $\tau$ here instead of $\b{\tau}$ as $\tau$ is irreducible.\\

We begin with a lemma. To state it; however, we need to introduce some notations. Given $\a,\beta\in F^\x$, let us put
$\h{t}(\a,\beta)=\varrho(m_2(t(\la,\beta)))$. This is an element in the diagonal torus of $\SO_5(F)$ 
(cf. \eqref{E:embedding}, \eqref{E:levi of Q} and \eqref{E:matrix})
According to \cite[Proposition 2.2]{Soudry1993}
\footnote{This is in fact a special case of more general results the asymptotic expansions of Whittaker functions 
(\cite{CasselmanShalika1980},\cite{LapidMao2009}).}
, there exists a finite set $\frak{X}_\pi$ of finite functions on 
$F^{\x,2}$, which depends only on (the class of) $\pi$, such that for $v\in\cV_\pi$
\[
W_{v}(\h{t}(\a\beta,\beta))
=
\sum_{\eta\in\frak{X}_\pi}\eta(\a,\beta)f_{v,\eta}(\a,\beta)
\]
for some $f_{v,\eta}\in\cS(F^2)$ (\cite[Proposition 2.2]{Soudry1993}). Since $\eta$ are finite functions, they can be further
written as
\[
\eta(\a,\beta)
=
\eta_1(\a)\eta_2(\beta)|\a|_F^{\frac{3}{2}}|\beta|_F^{2}(\log_{q^{-1}}|\a|_F)^{n_{\eta,1}}(\log_{q^{-1}}|\beta|_F)^{n_{\eta,2}}
\]
for some characters $\eta_1$, $\eta_2$ of $F^\x$ and non-negative integers $n_{\eta,1}$, $n_{\eta,2}$ by
\cite[Lemma 8.1]{JLbook}. 

\begin{lm}\label{L:upper bound}
Let $\pi$ be an irreducible generic representation of $\SO_5(F)$. Then we have 
\[
L(s,\pi\x\tau)\in L(2s,\omega_\tau)\prod_{\eta\in\frak{X}_\pi} L(2s,\eta_2\omega_\tau)^M\bbC[q^{-s},q^s]
\]
where $M=\max\stt{2^{n_{\eta,2}}\mid \eta\in\frak{X}_\pi}$.
\end{lm}

\begin{proof}
This is equivalent to show that 
\[
\Psi(v\ot\xi_s)\in L(2s,\omega_\tau)\prod_{\eta\in\frak{X}_\pi} L(2s,\eta_2\omega_\tau)^M\bbC[q^{-s},q^s]
\]
for every $v\in\cV_\pi$ and $\xi_s\in I^{\rm gd}(\tau,s)$. To prove this, we first analyze the poles of $\Psi(v\ot\xi_s)$ 
when $\xi_s\in I^{\rm std}(\tau,s)$. Using the Iwasawa decomposition $\SO_4(F)=V_2(F)\SO_4(\frak{o})$,  the right 
$\SO_4(\frak{o})$-finiteness and the fact that $\xi_s$ is a standard section, $\Psi(v\ot\xi_s)$ can be written as a finite 
sum of the integrals of the form
\[
\int_{F^\x}\int_{F^\x}
W_{v'}(\h{t}(\a\beta,\beta))W'_u(t(\a,1))\omega_\tau(\beta)|\a|_F^{s-2}|\beta|_F^{2s-2}d^\x\a d^\x\beta
\]
for some $v'\in\cV_\pi$ and $u\in\cV_\tau$. Since $\tau$ is supercuspidal, $W'_u(t(\a,1))$ is a Bruhat-Schwartz 
function on $F^\x$ (\cite[Lemma 14.3]{JLbook2}). This, together with the asymptotic expansion of 
$W_{v'}(\h{t}(\a\beta,\beta))$ and the fact that $\cS(F^2)\cong\cS(F)\ot_\bbC\cS(F)$, the integral can be further written as 
a finite sum of the integrals of the form
\begin{align*}
\int_{F^\x}&\int_{F^\x}
f_1(\a)f_2(\beta)\eta_1(\a)\eta_2(\beta)(\log_{q^{-1}}|\a|_F)^{n_{\eta,1}}(\log_{q^{-1}}|\beta|_F)^{n_{\eta,2}}
|\a|_F^{s-\frac{1}{2}}|\beta|_F^{2s}
d^\x\a d^\x\beta\\
&=
\left(
\int_{F^\x}
f_1(\a)\eta_1(\a)(\log_{q^{-1}}|\a|_F)^{n_{\eta,1}}|\a|_F^{s-\frac{1}{2}}d^\x\a
\right)
\left(
\int_{F^\x}
f_2(\beta)\eta_2\omega_\tau(\beta)(\log_{q^{-1}}|\beta|_F)^{n_{\eta,2}}|\beta|_F^{2s}d^\x\beta
\right)
\end{align*}
for some $f_1\in\cS(F^\x)$ and $f_2\in\cS(F)$, where $\cS(F^\x)$ and $\cS(F)$ stands for the space of Bruhat -Schwartz
functions on $F^\x$ and $F$ respectively. Clearly, the integral involving $f_1$ is always absolutely convergent (and 
hence has no poles). On the other hand, poles of the integral involving $f_2$ are contained in that of 
$\prod_{\eta\in\frak{X}_\pi} L(2s,\eta_2\omega_\tau)^M$. Consequently, we find that 
\[
\Psi(v\ot\xi_s)
\in
\prod_{\eta\in\frak{X}_\pi} L(2s,\eta_2\omega_\tau)^M\bbC[q^{-s},q^s]
\]
for every $v\in\cV_\pi$ and $\xi_s\in I^{\rm std}(\tau,s)$. From the definition of $I^{\rm hol}(\tau,s)$, it is clear that some 
assertion holds if we replace $\xi_s$ with any holomorphic section.\\

We also need to analyze the poles of $\Psi(v\ot\xi_s)$ when $\xi_s=M_\psi^\dagger(\tau^*,1-s)\xi^*_{1-s}$ for some 
$\xi^*_{1-s}\in I^{\rm hol}(\tau^*,1-s)$. For this, we use \cite[Theorem 5.1]{CasselmanShahidi1998} and \corref{C:LC} to 
deduce that
\[
L(2s,\omega_\tau)^{-1}M_\psi^\dagger(\tau^*,1-s)\xi^*_{1-s}\in I^{\rm hol}(\tau,s)
\]
for every $\xi^*_{1-s}\in I^{\rm hol}(\tau^*,1-s)$. It follows that 
\[
\Psi(v\ot M^\dagger_\psi(\tau^*,1-s)\xi^*_{1-s})
\in
L(2s,\omega_\tau)\prod_{\eta\in\frak{X}_\pi} L(2s,\eta_2\omega_\tau)^M\bbC[q^{-s},q^s].
\]
for every $\xi^*_{1-s}\in I^{\rm hol}(\tau^*,1-s)$. This completes the proof.
\end{proof}

\begin{Remark}
The characters $\eta_1,\eta_2$ appearing in the asymptotic expansions of the Whittaker functions of $\pi$ can be 
explicitly described by a result in the recent preprint of Chen in \cite{SYChen}.
\end{Remark}

\lmref{L:upper bound} has the following consequence.

\begin{cor}\label{C:temp}
\thmref{T:main} holds when $\pi$ is tempered.
\end{cor}

\begin{proof}
Since $\pi$ is tempered and $\omega_\tau$ is unitary, the poles of $L(2s,\eta_2\omega_\tau)$ are contained in the 
left half-plane $\Re(s)\le 0$ for every $\eta\in\frak{X}_\pi$.  This follows from the fact that the exponents of a tempered 
representations are non-negative (\cite[Proposition III.2.2]{Waldspurger2003}). 
By \lmref{L:upper bound}, we see that 
the $L$-factors $L(s,\pi\x\tau)$ and $L(1-s,\pi\x\tau^*)$ have no common poles. 
On the other hand, since the associated $L$-parameters are also tempered, same assertions hold for 
$L(s,\phi_\pi\ot\phi_\tau)$ and $L(1-s,\phi_\pi\ot\phi_{\tau^*})$. Now the corollary follows from the principle mentioned in
the beginning of this section.
\end{proof}

\subsection{Proof of \thmref{T:main}}
Now we are ready to prove \thmref{T:main}. In view of \thmref{T:theta}, it is nature to separate the proof into two cases, 
depending on whether $\phi_\pi$ is endoscopy or not. Here we follow the terminologies in 
\cite[Section 6]{GanTakeda2011}. In the proofs, we often regard $\pi$ as a representation of ${\rm GSp}_4(F)$ with trivial 
central character.

\subsubsection{Endoscopy case}
In this case, $\phi_\pi=\phi_{\tau_1}\oplus\phi_{\tau_2}$ for some irreducible generic representations $\tau_1$, $\tau_2$ 
of $\GL_2(F)$ with $\omega_{\tau_1}=\omega_{\tau_2}=1$. According to \cite[Section 7]{GanTakeda2011}, $\pi$ is 
the theta lift of the representation $\tau_1\bt\tau_2$ of ${\rm GSO}_4(F)$. On the other hand, by \thmref{T:theta}, it 
remains to prove the case when $\tau$ is isomorphic to an unramified twisted of $\tau_1^\vee$ or $\tau_2^\vee$. 
Assume without loss of generality that $\tau$ is isomorphic to an unramified twisted of $\tau_1^\vee$, so that 
$\tau_1$ is also supercuspidal. If $\tau_2$ is an essentially discrete series representation 
and $\tau_1$ and $\tau_2$ are not isomorphic, then $\pi$ is a discrete series representation by 
\cite[Theorem 5.6 (ii)]{GanTakeda2011}. In this case, \thmref{T:main} follows from \corref{C:temp}. If $\tau_1$
and $\tau_2$ are isomorphic, then $\pi$ is a tempered representation by 
\cite[Lemma 5.1 (a), Theorem 8.2 (i)]{GanTakeda2011b} and hence in this case, \thmref{T:main} again follows from
\corref{C:temp}.\\

Suppose that $\tau_2$ is an induced representation. Then since $\omega_{\tau_2}=1$, $\tau_2$ is inducing from the 
characters $\chi$ and $\chi^{-1}$ of $F^\x$. We may assume $|\chi(\varpi)|\ge 1$. Then since $\pi$ is generic, we must
have $\chi\neq |\cdot|_F^{-\frac{1}{2}}$ and $\pi\cong (\tau_1\ot\chi^{-1})\rtimes\chi$ by 
\cite[Lemma 5.2 (a), Theorem 8.2 (v)]{GanTakeda2011b}. Here we follow the notations in 
%\cite{SallyTadic1993}, 
\cite[Page 35]{RobertsSchmidt2007}, meaning that $\pi$ is a normalized induced representation of 
${\rm GSp}_4(F)$ inducing from its Siegel parabolic subgroup with the data $(\tau_1\ot\chi^{-1})\bt\chi$. 
Now by \lmref{L:upper bound} and \cite[Lemma 5.2, Type (X)]{SYChen}, we find that 
\[
L(s,\pi\x\tau)\in L(2s,\omega_\tau)^{M+1}\bbC[q^{-s},q^s].
\]
Since $\omega_\tau$ is assumed to be unitary, the $L$-factors $L(s,\pi\x\tau)$ and $L(1-s,\pi\x\tau^*)$ have no common
poles. On the other hand, since 
\[
\phi_\pi\ot\phi_\tau
\cong 
(\phi_{\tau_1}\oplus\phi_{\tau_2})\ot\phi_\tau
\cong
(\phi_{\tau_1}\ot\phi_{\tau})\oplus (\phi_{\tau_1}\ot\phi_{\chi})\oplus (\phi_{\tau_1}\ot\phi_{\chi^{-1}})
\]
as $8$-dimensional representations, we see that the poles of $L(s,\phi_\pi\ot\phi_\tau)$ are also lie in the vertical line 
$\Re(s)=0$. Consequently, the $L$-factors $L(s,\phi_\pi\ot\phi_\tau)$ and $L(1-s,\phi_\pi\ot\phi_{\tau^*})$ also have no 
common poles. This proves the \thmref{T:main} when $\phi_\pi$ is endoscopy.

\subsubsection{Non-endoscopy case}
In this case, the $\pi$ has a non-zero theta lift to split ${\rm GSO}_6(F)$, and the description of $\phi_\pi$ can be found in 
\cite[Lemma 6.2, Theorem 5.6 (iii)]{GanTakeda2011}. More precisely, $\phi_\pi$ is either an irreducible $4$-dimensional 
representation or $\phi_\pi=\phi\oplus\phi^\vee$ for some irreducible $2$-dimensional representation $\phi$, 
whose dual $\phi^\vee$ is not isomorphic $\phi$. In this first situation, $\pi$ is a discrete series representation and hence
we can apply \corref{C:temp} to conclude the proof.\\ 

Suppose that we are in the second situation. Then $\phi$ corresponds to an irreducible supercuspidal representation 
$\sigma$ of $\GL_2(F)$, and hence $\phi_\pi=\phi_\sigma\oplus\phi_{\sigma^\vee}$. Comparing this with
\cite[Table A.7 (VII)]{RobertsSchmidt2007}, we have
\[
\pi\cong\omega^{-1}_\sigma\rtimes\sigma
\]
with $\omega^{-1}_\sigma\neq 1$ and $\omega^{-1}_\sigma\neq\eta|\cdot|_F^{\pm 1}$, where $\eta$ is a quadratic 
character of $F^\x$ such that $\sigma\cong\sigma\ot\eta$. The notation means that $\pi$ is 
a normalized induced representation of ${\rm GSp}_4(F)$ inducing from its the Klingen subgroup with the data 
$\sigma\bt\omega^{-1}_{\sigma}$. We indicate that $\pi$ is not a discrete series representation, and is tempered if and 
only if $\omega_\sigma$ is unitary. The rest of the proof is devoted to verify \thmref{T:main} for this case.
First note that by \lmref{L:upper bound} and \cite[Lemma 5.2, (VII)]{SYChen}, we have
\[
L(s,\pi\x\tau)\in \cL(s,\pi\x\tau)\bbC[q^{-s},q^s]
\]
where 
\[
\cL(s,\pi\x\tau)
=
L(2s,\omega_\tau)L(2s,\omega_\tau\omega_\sigma)^M L(2s,\omega_\tau\omega^{-1}_\sigma)^M.
\]
On the other hand, since
\[
\phi_\pi\ot\phi_\tau
\cong
(\phi_\sigma\ot\phi_\tau)\oplus(\phi_{\sigma^\vee}\ot\phi_\tau)
\]
as $8$-dimensional representations, we have 
$L(s,\phi_\pi\ot\phi_\tau)=L(s,\phi_\sigma\ot\phi_\tau)L(s,\phi_{\sigma^\vee}\ot\phi_\tau)$.
The proof of the identity $L(s,\pi\x\tau)=L(s,\phi_\pi\ot\phi_\tau)$ consists of three steps.\\

The first step is to show 
\begin{equation}\label{E:1st step}
L(s,\phi_\pi\ot\phi_\tau)\in\cL(s,\pi\x\tau)\bbC[q^{-s},q^s].
\end{equation}
For this we first recall that by \cite[Proposition 1.2]{GelbartJacquet1978}, 
$L(s,\phi_\sigma\ot\phi_\tau)$ has a pole at $s_0\in\bbC$ if and only if $\sigma\cong\tau^\vee\ot|\cdot|_F^{-s_0}$, in
which case the pole is simple. The assertion being trivial when $L(s,\phi_\pi\ot\phi_\tau)=1$. So let us assume that 
$L(s,\phi_\pi\ot\phi_\tau)$ has poles. Then we have 
\[
\sigma\cong\tau^\vee\ot|\cdot|_F^{-s_0}
\quad\text{or}\quad
\sigma^\vee\cong\tau^\vee\ot|\cdot|_F^{-s_0^\vee}
\]
for some complex numbers $s_0$ and $s_0^\vee$. In the first case, since 
$\omega_\sigma\omega_\tau=|\cdot|_F^{-2s_0}$ by considering the central characters on both sides.
We find that
\[
L(s,\phi_\sigma\ot\phi_\tau)\in L(2s,\omega_\tau\omega_\sigma)^M\bbC[q^{-s},q^s].
\]
Similar argument shows 
\[
L(s,\phi_{\sigma^\vee}\ot\phi_\tau)\in L(2s,\omega_\tau\omega^{-1}_\sigma)^M\bbC[q^{-s},q^s]
\]
and the assertion follows.\\

Our second step is to show 
\begin{equation}\label{E:2nd step}
L(s,\pi\x\tau)\in L(s,\phi_\pi\ot\phi_\tau)\bbC[q^{-s},q^s].
\end{equation}
Note that by \eqref{E:1st step}, if $\cL(s,\pi\x\tau)$ and $\cL(1-s,\pi\x\tau^\vee)$ have no common poles, then 
$L(s,\pi\x\tau)=L(s,\phi_\pi\ot\phi_\tau)$. Note also that $\cL(s,\pi\x\tau)$ and 
$\cL(1-s,\pi\x\tau^\vee)$ have common poles would imply either $\omega_\sigma=|\cdot|_F^{\pm 2}$ or 
$\omega_\sigma^2=|\cdot|_F^{\pm 2}$. With these in mind, we now prove the second assertion. Certainly, it suffices to 
verify the assertion when $\cL(s,\pi\x\tau)$ and $\cL(1-s,\pi\x\tau^\vee)$ have common poles. For this we introduce 
an auxiliary $\la\in\bbR$ and consider the representation 
\[
\pi_{\la}
\cong
(\omega_\sigma^{-1}|\cdot|_F^{-2\la})\rtimes (\sigma\ot|\cdot|_F^{\la}).
\]
Note that the central character of $\pi_\la$ is again trivial. Let $\epsilon>0$ be small so that $\pi_\la$ remains irreducible 
for $|\la|<\epsilon$. Then for $0<|\la|<\epsilon$, we have $L(s,\pi_\la\x\tau)=L(s,\phi_{\pi_\la}\ot\phi_\tau)$ 
by the observations just mentioned. So for such $\la$, one has
\begin{equation}\label{E:2nd step'}
\Psi(v_\la\ot\xi_s)\in L(s,\phi_{\pi_\la}\ot\phi_\tau)\bbC[q^{-s},q^s]
\end{equation}
for every $v_\la\in\cV_{\pi_\la}$ and $\xi_s\in I^{\rm hol}(\tau,s)$. We would like to show this also holds when $\la=0$.
To do so, we use the holomorphicity of Whittaker functionals and the formal Laurent series. More precisely, by
\cite[Section 2]{CasselmanShalika1980} and \cite[Section 3]{Shahidi1978},  for a given $v\in\cV_\pi$, there exists for each
$\la$ a function $W_\la$ in the Whittaker model of $\pi_{\la}$ such that 
\begin{itemize} 
\item 
$W_0=W_v$;
\item 
there is an open compact subgroup $K$ such that $W_\la(gk)=W_\la(g)$ for all $g\in\SO_5(F)$, $k\in K$ and $\la$;
\item
for each $g\in\SO_5(F)$, we have $W_{\la}(g)\in\bbC[q^{-\la}, q^\la]$ as a function in $\la$.
\end{itemize}
Let $v_\la\in\cV_{\pi_\la}$ be such that its associated Whittaker function is $W_\la$. Then the second 
property implies that there exists $N>0$, which is independent of $\la$ such that 
\begin{equation}\label{E:vanish}
W_{v_\la}(\h{t}(\a\beta,\beta))=0
\end{equation}
if $\a$ or $\beta$ does not contain in $\varpi^{-N}\frak{o}$. For $\xi_s\in I^{{\rm std}}(\tau,s)$, the formal Laurent series 
$\Psi_{v_\la\ot\xi_s}(X)$ can be written as 
\[
\Psi_{v_{\la}\ot\xi_s}(X)
=
\sum_{m\in\bbZ}X^m\Psi^m(v_{\la}\ot\xi_s)
\]
with $\Psi^{m}(v_{\la}\ot\xi_s)$ given by \eqref{E:coeff of FLS}. Form this discussions in \S\ref{SS:Laurent series}, we 
know that $\Psi^m(v_{\la}\ot\xi_s)$ is independent of $s$. Moreover, by \eqref{E:vanish}, one sees that it can be written 
as a finite sum, each of which gives an element in $\bbC[q^{\la},q^\la]$ as a function of $\la$. 
It follows that $\Psi^m(v_{\la}\ot\xi_s)\in\bbC[q^{-\la},q^\la]$ as a function of $\la$ for all $m$.\\ 

To proceed, observe that $L(s,\phi_{\pi_\la}\ot\phi_\tau)=L(s+\la,\phi_\pi\ot\phi_\tau)$ for $|\la|<\epsilon$. 
Thus if $P(X)\in\bbC[X]$ is such that $P(0)=1$ and $L(s,\phi_\pi\ot\phi_\tau)=P(q^{-s})^{-1}$, then 
$L(s,\phi_\la\ot\phi_\tau)=P(q^{-\la-s})$. Now let us write 
\begin{equation}\label{E:finite sum}
P_{\pi\x\tau}(q^{-\la}X)\Psi_{v_\la\ot\xi_s}(X)=\sum_{m\in\bbZ}c_m(\la)X^m.
\end{equation}
Then we have $c_m(\la)\in\bbC[q^{-\la},q^\la]$ as a function of $\la$. As explained in 
\cite[(4.3)]{JPSS1983} and \cite[Section 8.2]{Kaplan2013}, \eqref{E:2nd step'} implies that for each $\la$ with 
$0<|\la|<\epsilon$, there is $N_\la>0$ such that $c_m(\la)=0$ whenever $|m|>N_\la$. We claim that there exists $N>0$ 
such that $c_m(\la)=0$ for all $|\la|<\epsilon$ and $|m|>N$. Indeed, since each non-zero $c_m(\la)$ 
has finitely many roots, while there are uncountable many $0<|\la|<\epsilon$, there must exists $N>0$ such that 
$c_m(\la)=0$ for all $0<|\la|<\epsilon$ and $|m|>N$. Since $c_m(\la)$ are continuous, the claim follows. The implies
that \eqref{E:finite sum} is a finite sum. In particular, by substituting $\la=0$ and $X=q^{-s}$ in \eqref{E:finite sum}, 
we get that 
\begin{equation}\label{E:la=0}
\Psi(v\ot\xi_s)\in L(s,\phi_\pi\ot\phi_\tau)\bbC[q^{-s}, q^s]
\end{equation}
for all $v\in\cV_\pi$ and $\xi_s\in I^{\rm std}(\tau,s)$, and hence for all $\xi_s\in I^{\rm hol}(\tau,s)$ by the definition of 
holomorphic sections.\\

To verify \eqref{E:2nd step}, it remains to show that \eqref{E:la=0} still valid if $\xi_s=M^\dagger_\psi(\tau,s)\xi^*_{1-s}$
for some $\xi^*_{1-s}\in I^{\rm hol}(\tau^\vee,1-s)$. But this follows immediately from the identity 
$\gamma(s,\pi\x\tau^\vee,\psi)=\gamma(s,\phi_\pi\ot\phi_{\tau^\vee},\psi)$ and \eqref{E:la=0} since
\[
\frac{\Psi(v\ot M_\psi^\dagger(\tau^\vee,1-s)\xi^*_{1-s})}
{L(s,\phi_\pi\ot\phi_\tau)}
=
\epsilon(s,\phi_\pi\ot\phi_{\tau^\vee},\psi)
\frac{\Psi(v\ot\xi^*_{1-s})}{L(1-s,\phi_\pi\ot\phi_{\tau^\vee})}\in\bbC[q^{-s}, q^s].
\]
This finishes the proof of our second step.\\

The last step is to show that 
$L(s,\pi\x\tau)=L(s,\phi_\pi\ot\phi_\tau)=L(s,\phi_\sigma\ot\phi_\tau)L(s,\phi_{\sigma^\vee}\ot\phi_\tau)$.
We divide the proof into two cases depending on whether $\tau^\vee$ is isomorphic to an unramified twist of $\tau$.
Suppose first that $\tau^\vee$ is not isomorphic to an unramified twist of $\tau$. In this case, we can not have both 
$\sigma$ and $\sigma^\vee$ are isomorphic to unramified twists of $\tau^\vee$. Consequently, the $L$-factors
$L(s,\phi_\pi\ot\phi_\tau)$ and $L(1-s,\phi_\pi\ot\phi_{\tau^\vee})$ have no common poles, and hence 
$L(s,\pi\x\tau)=L(s,\phi_\pi\ot\phi_\tau)$ by \eqref{E:2nd step}.\\

Suppose that $\tau^\vee$ is isomorphic to an unramified 
twist of $\tau$. Then by \eqref{E:shift} and because the similar identity holds for $L(s,\phi_\pi\ot\phi_\tau)$, we may
further assume that $\tau$ is self-dual, i.e. $\tau^\vee\cong\tau$. Now if $\sigma$ is not isomorphic to unramified twists of 
$\tau^\vee$, then we have $L(s,\phi_\pi\ot\phi_\tau)=1$ and the assertion follows again from \eqref{E:2nd step}. So let us
assume that $\sigma\cong\tau^\vee\ot|\cdot|_F^{-s_0}$ for some $s_0\in\bbC$. Then we have 
$\sigma^\vee\cong\tau\ot|\cdot|_F^{s_0}\cong\tau^\vee\ot|\cdot|_F^{s_0}$ and hence by
\cite[Corollary 1.3]{GelbartJacquet1978}
\[
L(s,\phi_\pi\ot\phi_\tau)
=
\left(1-q^{-\kappa(s-s_0)}\right)^{-1}\left(1-q^{-\kappa(s+s_0)}\right)^{-1}
\]
with $\kappa=2$ if $\tau\cong(\tau\ot\eta)$ and $\kappa=1$ otherwise, where $\eta$ is the unramfied quadratic character 
of $F^\x$. Similarly, we have 
\[
L(1-s,\phi_\pi\ot\phi_{\tau^\vee})
=
\left(1-q^{-\kappa(1-s-s_0)}\right)^{-1}\left(1-q^{-\kappa(1-s+s_0)}\right)^{-1}.
\]
We are done if we can show that these $L$-factors have no common poles. 
Suppose in contrary that they have common poles. Then from the
shapes of the $L$-factors, we must have 
\[
s_0=\pm\frac{1}{2}+\frac{2\pi n\sqrt{-1}}{\kappa\ln q}
\]
for some $n\in\bbZ$. It follows that $\omega^{-1}_\sigma=\omega^{-1}_\tau|\cdot|_F^{\pm 1}$ by the assumption
$\sigma^\vee\cong\tau^\vee\ot|\cdot|_F^{s_0}$. But this contradicts to the assumption on $\omega^{-1}_\sigma$ since
$\tau\ot\omega^{-1}_\tau\cong\tau^\vee\cong\tau$. This concludes the proof of the non-endoscopy case, and hence 
the proof of \thmref{T:main}.\qed

\subsection*{Acknowledgements}
Part of this works was done when the author was a postdoctoral fellow in the Institute of Mathematics at Academia Sinica. 
He would like to thank the institute for providing a very friendly and stimulating research atmosphere and Ming-Lun Hsieh 
for the support. Thanks also due to Shih-Yu Chen for the helpful discussions and to David Loeffler for pointing out his 
results. In particular, David Loeffler's comments help us simplify our original arguments and moreover, obtain 
assumption-free results. This work is partially supported by MOST grant number 110-2115-M-032-001-MY2.

\bibliographystyle{alpha}
\bibliography{ref}
\end{document}